\documentclass[11pt,reqno]{amsart} 

\makeatletter
\@namedef{subjclassname@2020}{\textup{2020} Mathematics Subject Classification}
\makeatother

\usepackage{amsmath,amssymb,amsthm,amscd}
\usepackage[T1]{fontenc}
\usepackage{blkarray} 

\usepackage{mathrsfs,euscript} 
\usepackage[cal=boondoxo,scr=euler]{mathalfa}
\usepackage{graphicx} 
\usepackage[colorlinks=true,urlcolor=blue,citecolor=red,linkcolor=blue,hyperfootnotes=true]{hyperref} 

\let\mc\mathcal

\let\eu\EuScript
\let\nc\newcommand

\newtheorem{thm}{Theorem}[section]
\newtheorem{cor}[thm]{Corollary}
\newtheorem{lem}[thm]{Lemma}
\newtheorem{prop}[thm]{Proposition}

\theoremstyle{definition}

\newtheorem{rem}[thm]{Remark}

\newtheorem{example}[thm]{Example}

\numberwithin{equation}{section}

\def\beq{\begin{equation}}
\def\eeq{\end{equation}}
\def\be{\begin{equation*}}
\def\ee{\end{equation*}}
\nc{\bea}{\begin{eqnarray*}}
\nc{\eea}{\end{eqnarray*}}

\let\al\alpha
\let\bt\beta

\let\Dl\Delta
\let\eps\varepsilon

\let\phi\varphi
\let\si\sigma

\let\thi\vartheta

\let\om\omega
\let\Om\Omega

\let\der\partial

\def\N{{\mathbb N}}
\def\C{{\mathbb C}}
\def\Z{{\mathbb Z}}
\def\Q{{\mathbb Q}}

\usepackage[OT2,T1]{fontenc}
\usepackage[all,cmtip]{xy} 
\usepackage{cancel}
\usepackage[vcentermath]{youngtab} 
\usepackage{empheq} 
\usepackage{changepage}
\usepackage{bm} 
\usepackage{chngcntr}
\counterwithin{figure}{section}
\counterwithin{table}{section}

\newlanguage\fakelanguage
\newcommand\cyr{\fontencoding{OT2}\fontfamily{wncyr}\selectfont
   \language\fakelanguage}
\DeclareTextFontCommand{\textcyr}{\cyr}

\usepackage{calligra}

\DeclareMathOperator{\HOM}{\mathscr{H}\text{\kern -3pt {\calligra\large om}}\,}

\makeatletter
\newsavebox{\@brx}
\newcommand{\llangle}[1][]{\savebox{\@brx}{$\m@th{#1\langle}$}%
  \mathopen{\copy\@brx\kern-0.5\wd\@brx\usebox{\@brx}}}
\newcommand{\rrangle}[1][]{\savebox{\@brx}{$\m@th{#1\rangle}$}%
  \mathclose{\copy\@brx\kern-0.5\wd\@brx\usebox{\@brx}}}
\makeatother

\newcommand\xqed[1]{%
  \leavevmode\unskip\penalty9999 \hbox{}\nobreak\hfill
  \quad\hbox{#1}}
\newcommand\qetr{\xqed{$\triangle$}}
\newcommand\qrem{\xqed{$\spadesuit$}}

\let\bi\bibitem

\usepackage{scalerel,stackengine}  
\stackMath
\newcommand\reallywidehat[1]{%
\savestack{\tmpbox}{\stretchto{%
  \scaleto{%
    \scalerel*[\widthof{\ensuremath{#1}}]{\kern.1pt\mathchar"0362\kern.1pt}%
    {\rule{0ex}{\textheight}}
  }{\textheight}%
}{2.4ex}}%
\stackon[-6.9pt]{#1}{\tmpbox}%
}
\usepackage{pifont}

\usepackage{booktabs}
\usepackage{adjustbox}  
\usepackage{lscape}     
\usepackage{longtable}  
\usepackage{bbm}
\usepackage{enumitem}

\begin{document}
\title[On the Betti numbers, Poincar\'e polynomials, and Euler characteristics of $\overline{\eu M}_{0,n}$]{On the Betti numbers, Poincar\'e polynomials, and Euler characteristics of $\overline{\eu M}_{0,n}$}
\author[Giordano Cotti]{Giordano Cotti$\>^{\circ}$}

{\let\thefootnote\relax
\footnotetext{\vskip5pt 
\noindent
$^\circ\>$\textit{ E-mail}:  giordano.cotti@tecnico.ulisboa.pt}}

\maketitle
\begin{center}
\textit{ 
$^\circ\>$Grupo de F\'isica Matem\'atica \\
Departamento de Matem\'atica, Instituto Superior T\'ecnico\\
Av. Rovisco Pais, 1049-001 Lisboa, Portugal\/\\}
\end{center}
\vskip2cm

\begin{abstract}
In this paper, we revisit the Poincar\'e polynomials, Betti numbers, and Euler characteristics of the Deligne--Mumford moduli spaces
\(\overline{\eu M}_{0,n}\)
of stable \(n\)-pointed rational curves. We give elementary derivations of two recent closed formulas for their Poincar\'e polynomials, due respectively to Aluffi--Marcolli--Nascimento and to Eur--Ferroni--Matherne--Pagaria--Vecchi. Our approach shows that both formulas are already implicit in the generating-series results of Getzler and Manin, and can be extracted from them by elementary manipulations of generating functions, the binomial series, and standard identities for Stirling numbers. Beyond these new derivations, the same method also yields new linear recurrence relations for refined invariants associated with these Poincar\'e polynomials, namely distinguished summands and a bivariate refinement. As a further consequence, we obtain two additional formulas for the Betti numbers, not previously recorded in this form.

We also study the Euler characteristics
\(\chi(\overline{\eu M}_{0,n})\).
Using the Taylor expansion of a suitable branch of the Lambert \(W\)-function, we show that their sequence is obtained by evaluating complete Bell polynomials at an explicit auxiliary integer sequence. This Bell-polynomial representation yields Hessenberg determinantal formulas and a new linear recursion, distinct from the well-known quadratic Keel--Manin recursion. It also provides an explicit extraction of the Euler characteristics from the Lambert \(W\)-function expression considered by Aluffi--Marcolli--Nascimento. Finally, we refine the Manin--Zagier asymptotic estimate for these Euler characteristics by computing the full asymptotic expansion.

\end{abstract}

\vskip0,3cm
\begin{adjustwidth}{35pt}{35pt}
{\footnotesize {\it Key words:} Moduli space, stable rational curves, Poincar\'e polynomial, generating function}
\vskip2mm
\noindent
{\footnotesize {\it 2020 Mathematics Subject Classification:} Primary: 14H10, 14D22; Secondary: 05A15, 05A16, 55N35 }
\end{adjustwidth}

\tableofcontents

\section{Introduction}

The Deligne--Mumford moduli spaces $\overline{\eu M}_{0,n}$, with $n \geq 3$, are smooth complex projective varieties of dimension $n-3$. They parametrize isomorphism classes of stable rational curves with $n$ marked points.\footnote{In genus zero, the fine moduli space agrees with its coarse moduli space; throughout the paper we use the notation $\overline{\eu M}_{0,n}$ for this space.} They play a central role in several areas of mathematics, from algebraic geometry to combinatorics, thanks to their concrete modular interpretation and their rich geometry and topology.

The geometry of these spaces admits several explicit descriptions. S.\,Keel~\cite{Kee92} described $\overline{\eu M}_{0,n}$ as an iterated blow-up of a product of projective lines along suitable augmented diagonals, while M.\,Kapranov~\cite{Kap93} gave a complementary blow-up construction. These descriptions place the genus-zero moduli spaces naturally in the framework of wonderful compactifications, a notion introduced by C.\,De Concini and C.\,Procesi~\cite{DCP95} and further developed by L.\,Li~\cite{Li09}.

Beyond their intrinsic geometric interest, the cohomology of \(\overline{\eu M}_{0,n}\) carries an exceptionally rich structure. It governs the geometry of Frobenius manifolds and the genus-zero sector of cohomological field theories, and is closely related to the theory of operads. For this reason, the topology and cohomology of these spaces have been the subject of sustained study over the last decades. See~\cite{Kee92,Man95,Get95,Yuz97,Man99} and the references therein.

Keel's work on the intersection theory of \(\overline{\eu M}_{0,n}\) gave a foundational description of its algebraic and topological invariants. In particular, he proved that the cycle class map identifies the Chow groups with homology groups,
so that the odd homology groups vanish and the integral Chow groups are finitely generated free abelian groups. He also obtained recursive formulas for the Betti numbers
and gave an explicit presentation of the Chow ring by generators and relations. Later, E.\,Getzler~\cite{Get95} and Yu.I.\,Manin~\cite{Man95} independently identified the generating series of the Poincar\'e polynomials
as the solution of an explicit Cauchy problem, or, equivalently, of an explicit functional equation. Despite these recursive and structural descriptions, a closed formula for the Poincar\'e polynomials was obtained only recently by P.\,Aluffi, M.\,Marcolli, and E.\,Nascimento, who discovered a strikingly elegant expression in terms of Stirling numbers of the first and second kinds~\cite[Thm.\,1.1]{AMN25}.

The original proof of the Aluffi--Marcolli--Nascimento formula is very ingenious but technically involved: it occupies several pages and relies on non-trivial combinatorial identities, including generalized forms of Vandermonde convolutions. Soon afterwards, an alternative proof was given by C.\,Eur, L.\,Ferroni, J.P.\,Matherne, R.\,Pagaria, and L.\,Vecchi~\cite{EFMPV25}. Their argument proceeds through a more general result. They first establish an explicit formula for the Hilbert series of the Chow ring of a polymatroid, with respect to an arbitrary building set. Specializing this formula to the braid matroid and the minimal building set, they recover the Poincar\'e polynomials of \(\overline{\eu M}_{0,n}\). In this specialization, the proof ultimately rests on Manin's and Getzler's earlier descriptions and on a remarkable combinatorial identity involving Stirling numbers; see~\cite[Formula~(8)]{EFMPV25}. As a further consequence of their Hilbert-series formula, they also obtain another explicit expression for these Poincar\'e polynomials; see~\cite[Thm.\,1.4]{EFMPV25}.

One of the purposes of this note is to give a short and elementary derivation of both the Aluffi--Marcolli--Nascimento formula \cite[Thm.\,1.1]{AMN25} and the Eur--Ferroni--Matherne--Pagaria--Vecchi formula \cite[Thm.\,1.4]{EFMPV25}. We show that both expressions are already implicit in the work of Getzler and Manin, and that they can be extracted from their results by straightforward manipulations of generating series. 
No combinatorial identities beyond the binomial series and standard identities for Stirling numbers are required. This gives a direct affirmative answer to a question raised by Aluffi--Marcolli--Nascimento \cite[Introduction]{AMN25} about whether their formula can be recovered from expressions following from Getzler's work\footnote{In contrast to \cite{EFMPV25}, our proof is a direct formal consequence of the Getzler--Manin generating series and does not use the polymatroidal machinery of their general theorem.}. Beyond these derivations, the method also reveals additional recursive structures: we obtain new linear recurrence relations for distinguished summands of the Poincar\'e polynomials and for a two-variable refinement. It also yields two explicit formulas for the individual Betti numbers, which appear not to have been recorded in this form.

We also obtain new results for the Euler characteristics of the spaces
\(\overline{\eu M}_{0,n}\).
We first show that the sequence of Euler characteristics admits a complete Bell-polynomial representation: more precisely, it is obtained by evaluating complete Bell polynomials at an explicit auxiliary sequence. Equivalently, this gives Hessenberg\footnote{An {\it upper/lower Hessenberg} matrix has zero entries below/above the first subdiagonal/superdiagonal.} determinantal expressions for
\(\chi(\overline{\eu M}_{0,n})\).
The derivation starts from the Taylor expansion of a suitable branch of the Lambert
\(W\)-function\footnote{In this way, we also answer a second question of Aluffi--Marcolli--Nascimento, who asked how to extract an explicit formula for the Euler characteristics from their Lambert \(W\)-function expression~\cite[Eq.~(8.1)]{AMN25}.}.  As a consequence of the Bell-polynomial representation, we obtain a new linear recursion for these Euler characteristics, distinct from the classical Keel--Manin recursion. Finally, we refine the Manin--Zagier asymptotic estimate~\cite{Man95} for
\(\chi(\overline{\eu M}_{0,n})\)
as \(n\to\infty\) by computing the full asymptotic expansion.

The spaces \(\overline{\eu M}_{0,n}\) form the genus-zero part of the broader system of moduli spaces \(\overline{\eu M}_{g,n}\) of stable pointed curves \cite{DM69,KM76,Knu83a,Knu83b}. In higher genus, the topology of these spaces is considerably more subtle: stable cohomology is governed by the Madsen--Weiss theorem \cite{MW07}, while many computations in the unstable range concern Euler characteristics, beginning with the work of Harer and Zagier \cite{HZ86} and continuing through contributions of Getzler, Loijenga \cite{Get95,GL99,Get02}, Manin \cite{Man95,Man99}, Bini--Gaiffi--Polito \cite{BGP01}, Faber--van der Geer \cite{FvdG04}, Bini--Harer \cite{BH11}, and others. Although our results concern only genus zero, they fit naturally into this general circle of questions.

\medskip

\noindent{\bf Acknowledgements.} The author wishes to thank D.\,Masoero, M.\,Mendes Lopes, L.\,Monsaingeon, J.\,Mourão, and A.\,Varchenko for valuable discussions. Special thanks are due to E.\,Russo and G.\,Ruzza for their careful reading of a preliminary draft. This research was supported by the Fundação para a Ciência e a Tecnologia (FCT)---Portuguese national funding through UID/00208/2025.

\section{Main results}

\noindent{\bf Notations. }
Throughout the paper, we use the convention
$
\mathbb N=\{1,2,3,\dots\}.
$
The exponential partial and complete Bell polynomials are denoted by
$
B_{n,m}$
and
$
B_n,
$
respectively; their definitions and the basic properties used in the paper are recalled in Appendix~\ref{appbell}. We use signed Stirling numbers of the first kind and Stirling numbers of the second kind, denoted by
$
s(n,k)$
and
$
S(n,k),
$
satisfying
$$
x^{\underline n}=x(x-1)\cdots(x-n+1)
=\sum_{k=0}^n s(n,k)x^k,\qquad 
x^n=\sum_{k=0}^n S(n,k)x^{\underline k}.
$$ We use the standard convention that these numbers vanish outside their natural range \(0\leq k\leq n\).
Finally, for
$
n\geq 3,
$
we set
$$
b_k(\overline{\eu M}_{0,n})
:=
\dim_{\mathbb C}H^k(\overline{\eu M}_{0,n},\mathbb C),\qquad 
{\sf P}_{\overline{\eu M}_{0,n}}(t)
:=
\sum_{k\geq 0}b_k(\overline{\eu M}_{0,n})t^k,\qquad
\chi(\overline{\eu M}_{0,n})
:=
{\sf P}_{\overline{\eu M}_{0,n}}(-1).
$$

\noindent{\bf Main results.} We now state the main results of the paper. They are organized around three themes.

The first concerns the Poincar\'e polynomials of
\(\overline{\eu M}_{0,n}\).
We introduce universal polynomials
\(\mathcal P_n\),
defined in terms of exponential partial Bell polynomials, and prove that
\({\sf P}_{\overline{\eu M}_{0,n+1}}(t)\)
is obtained by evaluating
\(\mathcal P_n\)
at an explicit binomial sequence
\(\pi_k(t)\), see Theorem \ref{mainthm1}.
After rewriting, this formula is equivalent to the specialization of the Eur--Ferroni--Matherne--Pagaria--Vecchi formula to
\(\overline{\eu M}_{0,n+1}\), see Corollary \ref{corEFMPV}.
Our proof is instead a direct formal consequence of the Getzler--Manin generating series and does not use the polymatroidal machinery of their general theorem. In addition to giving a new proof of this formula, our method also produces genuinely new linear recurrence relations for refined invariants associated with \({\sf P}_{\overline{\eu M}_{0,n}}(t)\).
These refinements include both distinguished summands of the Poincaré polynomial and a two-variable refined version, see Theorem \ref{THMrecP}. We then derive two explicit formulas for the Betti numbers of $\overline{\eu M}_{0,n}$, see Theorems \ref{mainthm2} and \ref{betti2}.

The second contribution is a direct elementary derivation of the Aluffi--Marcolli--Nascimento formula, see Secction \ref{AMNsec}. The argument again starts from the Getzler--Manin generating series; the closed formula follows from elementary manipulations of exponential generating functions, the binomial series, and the standard Stirling-number expansion of falling factorials. Thus both recent closed formulas for the Poincar\'e polynomials are recovered from the Getzler--Manin framework.

The third theme concerns Euler characteristics. We show that the sequence $(\chi(\overline{\eu M}_{0,n+2}))_{n\geq 1}$ can be recovered from the sequence
\[\frac{0}{2},\frac{1}{3},\frac{2}{4},\frac{3}{5},\frac{4}{6},\frac{5}{7},\frac68,\frac79,\frac{8}{10},\dots
\]
by a two-step transform. 
More precisely, we construct the transformation
$$
\left(\frac{n-1}{n+1}\right)_{n\geq 1}
\longmapsto
\left(\thi_n\right)_{n\geq 1}
\longmapsto
\left(\chi(\overline{\eu M}_{0,n+2})\right)_{n\geq 1},
$$
where 
the intermediate sequence \((\thi_n)_{n\geq 1}\) is obtained by evaluating the universal polynomials \(\mathcal P_n\),
while the second arrow represent the complete Bell-polynomials transform, see Theorem \ref{mainthm}. This automatically gives Hessenberg determinantal formulas for the Euler characteristics, as well as a new linear recursion distinct from the quadratic Keel--Manin recursion, \ref{corlinearrechi}. Finally, using the Lambert
$W$-function, we refine the Manin--Zagier asymptotic estimate by explicitly giving the full asymptotic expansion, see Theorem \ref{thmMZ+}.

\subsection{The $\mc P$-polynomials} Let $\bm x=(x_1,x_2,\dots)$ be indeterminates. For any natural number $n\geq 2$ define the polynomial $\mc P_n\in\Q[x_1,\dots,x_{n-1}]$ by
\[\mc P_n(\bm x):=\sum_{m=1}^{n-1}\frac{(n+m-1)!}{(n-1)!}B_{n-1,m}(x_1,x_2,\dots).
\]

\begin{example}The first $\mc P$-polynomials are
\[
\begin{aligned}
\mathcal P_2(\bm x)&=2x_1,\\
\mathcal P_3(\bm x)&=12x_1^2+3x_2,\\
\mathcal P_4(\bm x)&=120x_1^3+60x_1x_2+4x_3,\\
\mathcal P_5(\bm x)&=
1680x_1^4
+1260x_1^2x_2
+120x_1x_3
+90x_2^2
+5x_4.
\end{aligned}
\]\qetr
\end{example}

\begin{prop}
For every integer $n \ge 2$, $\mc P_n(\bm x)$ is a polynomial in $\Z[x_1,\dots,x_{n-1}]$, and all its coefficients are positive multiples of $n$.
\end{prop}

\proof
Incomplete Bell polynomials have positive integer coefficients. Each ratio $\frac{(n+m-1)!}{(n-1)!}$, with $m=1,\dots,n-1$ and $n\geq 2$, is a positive integer divisible by $n$. The claim follows.
\endproof

Let us decompose each polynomial $\mc P_n(\bm x)$ into homogeneous parts, by setting
\[
\mc P_n(\bm x)
=
\sum_{m=1}^{n-1}\mc P_{n,m}(\bm x),
\qquad
\mc P_{n,m}(\bm x)
:=
\frac{(n+m-1)!}{(n-1)!}
B_{n-1,m}(x_1,x_2,\dots),\qquad 
0< m<n.
\]
The polynomial $\mc P_{n,m}(\bm x)$ is homogeneous of degree $m$ in the variables $\bm x$.

\begin{prop}
The polynomials $\mc P_{n,m}$ satisfy the recurrence
\beq\label{recP1}
\mc P_{n+1,m+1}
=
\sum_{j=1}^{n-m}
\frac{(n+m+1)!}
{n\,(j-1)!\,(n-j+m)!}
\,x_j\,
\mc P_{n-j+1,m}.
\eeq
If we set $\mc P_1:=1$
and
$
E:=\sum_{r\geq 1}x_r\frac{\partial}{\partial x_r}
$
is the Euler operator, then the complete polynomials satisfy
\beq\label{recP2}
\mc P_{n+1}
=
\sum_{j=1}^{n}
\frac{x_j}{n\,(j-1)!}
\left[
\prod_{a=0}^{j}
\left(E+n-j+1+a\right)
\right]
\mc P_{n-j+1}.
\eeq
\end{prop}
\proof
Using the standard recurrence for the partial exponential Bell polynomials,
$$
B_{n,m+1}
=
\sum_{i=0}^{n-m-1}
\binom{n-1}{i}
x_{i+1}
B_{n-i-1,m},
$$
and multiplying by
$
\frac{(n+m+1)!}{n!},
$
we obtain
$$
\mc P_{n+1,m+1}
=
\sum_{i=0}^{n-m-1}
\frac{(n+m+1)!}{n!}
\binom{n-1}{i}
x_{i+1}
B_{n-i-1,m}
=
\sum_{i=0}^{n-m-1}\frac{(n+m+1)!}{n\,i!\,(n-i+m-1)!}.
x_{i+1}
\mc P_{n-i,m}.
$$This proves \eqref{recP1}, by setting $j=i+1$. From the identity $E\mc P_{r,m}=m\mc P_{r,m}$, we deduce
$$
\prod_{a=0}^{j}
\left(E+n-j+1+a\right)\mc P_{r,m}
=
\prod_{a=0}^{j}
\left(m+n-j+1+a\right)\mc P_{r,m}
=
\frac{(n+m+1)!}{(n-j+m)!}\mc P_{r,m}.
$$
Therefore
$$
\frac{(n+m+1)!}
{n\,(j-1)!\,(n-j+m)!}
\mc P_{n-j+1,m}
=
\frac{1}{n\,(j-1)!}
\left[
\prod_{a=0}^{j}
\left(E+n-j+1+a\right)
\right]
\mc P_{n-j+1,m}.
$$
Summing this identity over all admissible $m$ and $j$
gives \eqref{recP2}.
\endproof

\subsection{The $\pi$-sequence and Poincar\'e polynomials} Consider the sequence $\pi\in\Q[t]^{\N}$ defined by
\beq
\label{pik}\pi_k(t)=\frac{(k-1)!}{k+1}\binom{t^2-2}{k-1}=
\frac{(t^2-2)(t^2-3)\cdots(t^2-k)}{k+1},\quad k\geq 1,
\eeq with the empty product convention for $k=1.$

\begin{thm}\label{mainthm1}
For all $n\geq 2$, we have
\[{\sf P}_{\overline{\eu M}_{0,n+1}}(t)=\mc P_n(\pi_1(t),\dots,\pi_{n-1}(t)).
\]
\end{thm}

Theorem \ref{mainthm1} will be proved in Section \ref{secproof1}.

For a finite set
$I,$
we write
$
\lambda\vdash I
$
for a set partition of
$
I,
$
and
$
\ell(\lambda)
$
for its number of blocks. We also set $[n-1]=\{1,\dots,n-1\}$.
\begin{cor}\cite[Thm.\,1.4]{EFMPV25}\label{corEFMPV}
We have
$$
{\sf P}_{\overline{\eu M}_{0,n+1}}(t)
=
\sum_{\lambda\vdash [n-1]}
\frac{(n-1+\ell(\lambda))!}{(n-1)!}
\prod_{B\in\lambda}
\frac{\prod_{j=2}^{|B|}(t^2-j)}{|B|+1},
$$
where the product
is understood to be empty when
$
|B|=1.
$
\end{cor}

\begin{proof}
By the set-partition expansion of the partial Bell polynomials (formula \eqref{eqbellpart}),
\[
B_{N,m}(y_1,y_2,\dots)
=
\sum_{\substack{\lambda\vdash [N]\\ \ell(\lambda)=m}}
\prod_{B\in\lambda} y_{|B|}.
\]
Applying this identity with \(N=n-1\) and \(y_k=\pi_k(t)\), Theorem~\ref{mainthm1} gives
\[
{\sf P}_{\overline{\eu M}_{0,n+1}}(t)
=
\sum_{\lambda\vdash [n-1]}
\frac{(n-1+\ell(\lambda))!}{(n-1)!}
\prod_{B\in\lambda}\pi_{|B|}(t).
\]
The formula follows from the definition of \(\pi_k(t)\).
\end{proof}

\begin{example}
Throughout the example, we write $\pi_i$ for $\pi_i(t)$.
For instance, we have
\begin{align*}
{\sf P}_{\overline{\eu M}_{0,3}}(t)
  &= 2 \pi_1
   = 1,\\
{\sf P}_{\overline{\eu M}_{0,4}}(t)
  &= 12 \pi_1^2 + 3 \pi_2
   = 1+t^2,\\
{\sf P}_{\overline{\eu M}_{0,5}}(t)
  &= 120 \pi_1^3 + 60 \pi_2 \pi_1 + 4 \pi_3
   = t^4+5t^2+1,
\end{align*}
and also
\begin{align*}
{\sf P}_{\overline{\eu M}_{0,6}}(t)
  ={}& 1680 \pi_1^4
      +1260 \pi_2 \pi_1^2
      +120 \pi_3 \pi_1
      +90 \pi_2^2
      +5 \pi_4 \\
  ={}& 1 + 16t^2 + 16t^4 + t^6,
\end{align*}
while
\begin{align*}
{\sf P}_{\overline{\eu M}_{0,11}}(t)
={}&
17643225600 \pi_1^9
+35286451200 \pi_2 \pi_1^7
+4843238400 \pi_3 \pi_1^6
+21794572800 \pi_2^2 \pi_1^5\\
&+454053600 \pi_4 \pi_1^5
+4540536000 \pi_2 \pi_3 \pi_1^4
+30270240 \pi_5 \pi_1^4
+4540536000 \pi_2^3 \pi_1^3\\
&+201801600 \pi_3^2 \pi_1^3
+302702400 \pi_2 \pi_4 \pi_1^3
+1441440 \pi_6 \pi_1^3
+908107200 \pi_2^2 \pi_3 \pi_1^2\\
&+21621600 \pi_3 \pi_4 \pi_1^2
+12972960 \pi_2 \pi_5 \pi_1^2
+47520 \pi_7 \pi_1^2
+227026800 \pi_2^4 \pi_1\\
&+43243200 \pi_2 \pi_3^2 \pi_1
+415800 \pi_4^2 \pi_1
+32432400 \pi_2^2 \pi_4 \pi_1
+665280 \pi_3 \pi_5 \pi_1\\
&+332640 \pi_2 \pi_6 \pi_1
+990 \pi_8 \pi_1
+369600 \pi_3^3
+21621600 \pi_2^3 \pi_3\\
&+1663200 \pi_2 \pi_3 \pi_4
+498960 \pi_2^2 \pi_5
+13860 \pi_4 \pi_5
+9240 \pi_3 \pi_6\\
&+3960 \pi_2 \pi_7
+10 \pi_9\\
={}&
t^{16}+968t^{14}+49556t^{12}+429594t^{10}+861235t^8
+429594t^6+49556t^4+968t^2+1.\tag*{\qetr}
\end{align*}
\end{example}

We next recall the closed formula of Aluffi--Marcolli--Nascimento for the Poincar\'e polynomials ${\sf P}_{\overline{\eu M}_{0,n+1}}(t)$.

\begin{thm}\label{AMNthm}\cite[Thm.\,1.1]{AMN25} We have
\begin{multline*}{\sf P}_{\overline{\eu M}_{0,n+1}}(t)=(1-t^2)^n\sum_{a\ge0} \sum_{b\ge 0}
s(n+a,n+a-b)\, S(n+a-b,a+1)\,
t^{2(a+b)}\\
=(1-t^2)^n(-1)^n
\sum_{k\ge0} \sum_{r=0}^{n+k-1}
s(n+k-1,r)\, S(r,k)\,
t^{2(r-n-2k)},
\end{multline*}where the second equality follows by Poincar\'e duality ${\sf P}_{\overline{\eu M}_{0,n+1}}(t)=t^{2\dim_\C\overline{\eu M}_{0,n+1}}{\sf P}_{\overline{\eu M}_{0,n+1}}(t^{-1})$.
\end{thm}

Our contribution here is the elementary proof given in Section~\ref{AMNsec}.

\subsection{An auxiliary recurrence for refined Poincar\'e polynomials} The Poincar\'e polynomials ${\sf P}_{\overline{\eu M}_{0,n}}(t)$ satisfy a well-known quadratic recurrence\footnote{Here we use the convention
$
{\sf P}_{\overline{\eu M}_{0,2}}(t):=1,
$
so that the recurrence is valid starting from $n=1$. Equivalently,
one may avoid this unstable convention by writing, for $n\geq 2$,
$$
{\sf P}_{\overline{\eu M}_{0,n+2}}(t)
=
(1+t^2)
{\sf P}_{\overline{\eu M}_{0,n+1}}(t)
+
t^2
\sum_{\substack{i+j=n+1\\ i,j\geq 2}}
\binom{n}{i}
{\sf P}_{\overline{\eu M}_{0,i+1}}(t)\,
{\sf P}_{\overline{\eu M}_{0,j+1}}(t),
$$
with initial value
$
{\sf P}_{\overline{\eu M}_{0,3}}(t)=1.
$}, due to S.\,Keel and Yu.I.\,Manin \cite{Kee92,Man95}:
\[{\sf P}_{\overline{\eu M}_{0,n+2}}(t)={\sf P}_{\overline{\eu M}_{0,n+1}}(t)+t^2\sum_{\substack{i+j=n+1\\ i\geq 2}}\binom{n}{i}{\sf P}_{\overline{\eu M}_{0,i+1}}(t)\,{\sf P}_{\overline{\eu M}_{0,j+1}}(t), \quad n\geq 1.
\]In this section we derive alternative {\it linear} recurrence relations for certain refinements of the Poincaré polynomial ${\sf P}_{\overline{\eu M}_{0,n}}(t)$. More precisely, these recurrences concern both distinguished summands of the polynomial and a bivariate refinement.
Namely, for $0<m<n$, set
\[{\sf P}_{\overline{\eu M}_{0,n+1}}(t)=\sum_{m=1}^{n-1}{\sf P}_{n,m}(t),\quad {\sf P}_{n,m}(t):=\mc P_{n,m}(\pi_1(t),\pi_2(t),\dots),
\]and, for $n\geq 2$, introduce also the refined polynomials
\[{\sf P}_n(s,t)=\mc P_n(s\,\pi_1(t),s\,\pi_2(t),\dots)=
\sum_{m=1}^{n-1}s^m {\sf P}_{n,m}(t),\qquad \text{so that }{\sf P}_n(1,t)={\sf P}_{\overline{\eu M}_{0,n+1}}(t).
\]
\begin{thm}\label{THMrecP}
The polynomials ${\sf P}_{n,m}$ satisfy the closed recurrence
$$
{\sf P}_{n+1,m+1}(t)
=
\frac{(n+m+1)!}{n}
\sum_{j=1}^{n-m}
\frac{1}{(j+1)(n-j+m)!}
\binom{t^2-2}{j-1}
{\sf P}_{n-j+1,m}(t).
$$
Consequently,
$$
{\sf P}_{\overline{\eu M}_{0,n+2}}(t)
=
\sum_{m=0}^{n-1}
\frac{(n+m+1)!}{n}
\sum_{j=1}^{n-m}
\frac{1}{(j+1)(n-j+m)!}
\binom{t^2-2}{j-1}
{\sf P}_{n-j+1,m}(t).
$$
Equivalently, if we set ${\sf P}_1=1,$ the refined polynomials ${\sf P_n}$ satisfy the closed recurrence
$$
{\sf P}_{n+1}(s,t)
=
\sum_{j=1}^{n}
\frac{s}{n(j+1)}
\binom{t^2-2}{j-1}
\left[
\prod_{a=0}^{j}
\left(
s\frac{\partial}{\partial s}+n-j+1+a
\right)
\right]
{\sf P}_{n-j+1}(s,t).
$$
\end{thm}
\proof
The result directly follows from identities \eqref{recP1} and \eqref{recP2}, by specializing $x_i=\pi_i(t)$.
\endproof

\begin{example}
We illustrate the recurrence for the refined polynomials in the case
corresponding to the computation of ${\sf P}_6(s,t)$.  Put
$$
u:=t^2,
\qquad
\Delta:=s\frac{\partial}{\partial s}.
$$
The previous refined polynomials are
$$
\begin{aligned}
{\sf P}_1(s,t)&=1,\\
{\sf P}_2(s,t)&=s,\\
{\sf P}_3(s,t)&=s(u-2)+3s^2,\\
{\sf P}_4(s,t)&=s(u-2)(u-3)+10s^2(u-2)+15s^3,\\
{\sf P}_5(s,t)&=
s(u-2)(u-3)(u-4)
+5s^2(u-2)(5u-13)
+105s^3(u-2)+105s^4.
\end{aligned}
$$
The recurrence, applied with $n=5$, gives
$$
\begin{aligned}
{\sf P}_6(s,t)
={}&
\frac{s}{10}
(\Delta+5)(\Delta+6)\,{\sf P}_5(s,t)\\
&+
\frac{s}{15}
(u-2)
(\Delta+4)(\Delta+5)(\Delta+6)\,{\sf P}_4(s,t)\\
&+
\frac{s}{20}
\binom{u-2}{2}
(\Delta+3)(\Delta+4)(\Delta+5)(\Delta+6)\,{\sf P}_3(s,t)\\
&+
\frac{s}{25}
\binom{u-2}{3}
(\Delta+2)(\Delta+3)(\Delta+4)(\Delta+5)(\Delta+6)\,{\sf P}_2(s,t)\\
&+
\frac{s}{30}
\binom{u-2}{4}
(\Delta+1)(\Delta+2)(\Delta+3)(\Delta+4)(\Delta+5)(\Delta+6)\,{\sf P}_1(s,t).
\end{aligned}
$$
Evaluating the right-hand side, we obtain
$$
\begin{aligned}
{\sf P}_6(s,t)
={}&
s(u-2)(u-3)(u-4)(u-5)\\
&+
14s^2(u-2)(u-3)(4u-11)\\
&+
70s^3(u-2)(7u-17)\\
&+
1260s^4(u-2)
+
945s^5.
\end{aligned}
$$
In particular, setting $s=1$, one recovers
\[
{\sf P}_6(1,t)={\sf P}_{\overline{\eu M}_{0,7}}(t)
=
t^8+42t^6+127t^4+42t^2+1.
\tag*{\qetr}
\]
\end{example}

\subsection{Two formulas for Betti numbers}
Aluffi--Marcolli--Nascimento give two closed formulas for the Betti numbers in~\cite{AMN25}. The first is deduced directly from their Poincar\'e polynomial formula and expresses the Betti numbers in terms of Stirling numbers of the first and second kinds; see~\cite[Cor.\,1.2]{AMN25}. The second expresses them in terms of Bernoulli numbers; see~\cite[Cor.\,1.5]{AMN25}.

In this section we give two alternative formulas.

The first one is obtained by extracting coefficients from the identity in Theorem~\ref{mainthm1}.

For $k\geq 1$, write
\[
\pi_k(t)=\sum_{a=0}^{k-1} \bt_{k,a}\,t^{2a},
\qquad \bt_{k,a}\in\mathbb{Q}.
\]
\begin{lem}
We have
\begin{equation}\label{eq:cka}
\bt_{k,a}
=
\frac{1}{k+1}
\sum_{q=a}^{k-1}
(-2)^{q-a}\binom{q}{a}
s(k-1,q)
,
\qquad 0\le a\le k-1.
\end{equation}
\end{lem}

\proof
A direct consequence of the identity \(
x^{\underline{m}}=x(x-1)\dots(x-m+1)=\sum_{q=0}^{m}s(m,q)x^q,
\)
and the binomial expansion for $(t^2-2)^q$.
\endproof
Combining the lemma above with the following theorem gives an explicit closed formula for all coefficients of
\({\sf P}_{\overline{\eu M}_{0,n+1}}(t)\).

\begin{thm}\label{mainthm2}For all $0\leq \ell\leq n-2$, we have
\[{b}_{2\ell}(\overline{\eu M}_{0,n+1})=
\sum_{\substack{
r_{k,a}\ge 0\\
k=1,\dots,n-1\\
a=0,\dots,k-1\\
\sum\limits_{k,a} k\, r_{k,a} = n-1 \\
\sum\limits_{k,a} a\, r_{k,a} = \ell}
}
\frac{(n + \sum\limits_{k,a} r_{k,a} - 1)!}
{\prod\limits_{k=1}^{n-1} (k!)^{\sum\limits_{a=0}^{k-1} r_{k,a}}}
\;
\prod_{k=1}^{n-1} \prod_{a=0}^{k-1}
\frac{\bt_{k,a}^{\,r_{k,a}}}{r_{k,a}!}.
\]
\end{thm}
 
 Theorem \ref{mainthm2} will be proved in Section \ref{proofbetti1}.
 
We now state a second closed formula. To this end, we introduce two auxiliary quantities.
For any $a,b,c \in \mathbb{Z}_{\geq 0}$, define
\begin{multline*}
\Om(a,b,c):=\\=\sum_{\substack{i_1+\dots+i_{c+2}=a\\ i_1,\dots,i_{c+2}\ge 0}}(-1)^{\sum_{r=1}^{c+1}i_r}\binom{a}{i_1,\dots,i_{c+2}}
\frac{1}{\prod_{j=1}^{c+1}j!^{i_j}}\frac{\left(\sum_{j=1}^{c+1}ji_j\right)!\, b!}{(\sum_{r=1}^{c+1}i_r+b)!}\, s\left(\sum_{r=1}^{c+1}i_r+b,\, \sum_{j=1}^{c+1}ji_j\right),
\end{multline*}
and
\[
\Omega^{[1]}(a,b,c)
:= \frac{a!}{b!}
\sum_{\substack{
\sum_{l=1}^{a-b+1} h_l = b \\
\sum_{l=1}^{a-b+1} l h_l = a
}}
\sum_{\substack{
\sum_{l=1}^{a-b+1} \sigma_l = c
}}
\binom{b}{h_1,h_2,\dots}\binom{c}{\sigma_1,\sigma_2,\dots}
\prod_{l=1}^{a-b+1}
\Omega(h_l,\sigma_l,l),
\]with the convention $\Om^{[1]}(a,b,c)=0$ if $b>a$.

The following theorem gives an alternative formula for Betti numbers in terms of $\Om(a,b,c)$ and $\Om^{[1]}(a,b,c)$.

\begin{thm}\label{betti2}
We have
\[b_{2\ell}(\overline{\eu M}_{0,n+1})=\sum_{m=0}^\ell\sum_{i=0}^{n-1}\sum_{k=m}^{m+i}\binom{n+k-1}{i,\,n-i-1,\,m,\,k-m}
\,\Om(k-m,i,0)\cdot \Om^{[1]}(\ell,m,n-1-i).
\]
\end{thm}

Theorem \ref{betti2} will be proved in Section \ref{proofbetti2}.

\subsection{The $\thi$-sequence and Euler characteristics} Set $\eps_k := \frac{k-1}{k+1}$ for any $k\in\N$.
Consider the sequence $\thi\in\Q^{\N}$ defined by
\begin{equation}\label{eq:def-bn}
\thi_1 := 1,\qquad \thi_n
:= \mc P_n(\eps_1,\dots,\eps_{n-1}), 
\quad n\geq 2.
\end{equation}
The first terms of the $\thi$-sequence are
\begin{multline*}1, \quad 0,\quad 1,\quad 2,\quad 13,\quad 74,\quad 593,\quad 5298,\quad 55781,\quad 659722,\quad 8745185,\quad 127941538,\\ 
2050576669,\quad 35712158442,\quad 671672584945,\dots
\end{multline*}

The integrality suggested by the initial values holds in general.
\begin{prop}\label{integrprop}
For all $n \geq 1$, one has $\thi_n \in \mathbb{Z}_{\geq 0}$.
\end{prop}

\begin{thm}\label{mainthm}
For all $n\geq 2$, we have
\[\chi(\overline{\eu M}_{0,n+1})=B_{n}(\thi_1,\thi_2,\dots,\thi_n).\]
Equivalently, we have
\begin{multline*}\chi(\overline{\eu M}_{0,n+1})=\det \begin{pmatrix}
\thi_{1} & \binom{n-1}{1}\thi_{2} & \binom{n-1}{2}\thi_{3} & \binom{n-1}{3}\thi_{4} & \cdots & \cdots & \thi_{n} \\
-1   & \thi_{1}               & \binom{n-2}{1}\thi_{2} & \binom{n-2}{2}\thi_{3} & \cdots & \cdots & \thi_{n-1} \\
0    & -1                  & \thi_{1}               & \binom{n-3}{1}\thi_{2} & \cdots & \cdots & \thi_{n-2} \\
0    & 0                   & -1                  & \thi_{1}               & \cdots & \cdots & \thi_{n-3} \\
0    & 0                   & 0                   & -1                  & \cdots & \cdots & \thi_{n-4} \\
\vdots & \vdots            & \vdots              & \vdots              & \ddots & \ddots & \vdots \\
0    & 0                   & 0                   & 0                   & \cdots & -1     & \thi_{1}
\end{pmatrix}\\
=\det\begin{pmatrix}
\frac{\thi_{1}}{0!} & \frac{\thi_{2}}{1!} & \frac{\thi_{3}}{2!} & \frac{\thi_{4}}{3!} & \cdots & \cdots & \frac{\thi_{n}}{(n-1)!} \\
-1   & \frac{\thi_{1}}{0!} & \frac{\thi_{2}}{1!} & \frac{\thi_{3}}{2!} & \cdots & \cdots & \frac{\thi_{n-1}}{(n-2)!} \\
0    & -2                  & \frac{\thi_{1}}{0!}               & \frac{\thi_{2}}{1!} & \cdots & \cdots & \frac{\thi_{n-2}}{(n-3)!} \\
0    & 0                   & -3                  & \frac{\thi_{1}}{0!}               & \cdots & \cdots & \frac{\thi_{n-3}}{(n-4)!} \\
0    & 0                   & 0                   & -4                  & \cdots & \cdots & \frac{\thi_{n-4}}{(n-5)!} \\
\vdots & \vdots            & \vdots              & \vdots              & \ddots & \ddots & \vdots \\
0    & 0                   & 0                   & 0                   & \cdots & -(n-1)     & \frac{\thi_{1}}{0!}
\end{pmatrix}.
\end{multline*}
\end{thm}

\begin{example}
For instance, we have\footnote{Notice that the identity $\thi_2=0$ simplifies considerably the expression of the complete Bell polynomial.}
\begin{multline*}
\chi(\overline{\eu M}_{0,10})=\thi_1^9+84 \thi_3 \thi_1^6+126 \thi_4 \thi_1^5+126 \thi_5 \thi_1^4+840 \thi_3^2 \thi_1^3+84 \thi_6 \thi_1^3+1260 \thi_3 \thi_4 \thi_1^2\\+36 \thi_7 \thi_1^2+315 \thi_4^2 \thi_1+504 \thi_3 \thi_5 \thi_1+9 \thi_8 \thi_1+280 \thi_3^3+126 \thi_4 \thi_5+84 \thi_3 \thi_6+\thi_9
\\ =153946.\tag*{\qetr}
\end{multline*}
\end{example}

Proposition \ref{integrprop} and Theorem \ref{mainthm} will be proved in Sections \ref{proofintegrprop} and \ref{proofmainthm}, respectively.

\begin{cor}\label{corlinearrechi}
For all $n\geq 0$, we have 
\beq\label{rec1} \chi(\overline{\eu M}_{0,n+2})=\sum_{k=0}^n\binom{n}{k}\thi_{k+1}\chi(\overline{\eu M}_{0,n-k+1}),
\eeq where we set $\chi(\overline{\eu M}_{0,1})=\chi(\overline{\eu M}_{0,2})=1$.
\end{cor}
\proof
This directly follows from the recursive definition of complete Bell polynomials
\[B_0:=1,\qquad B_{n+1}(x_1,\dots,x_{n+1})=\sum_{k=0}^n\binom{n}{k}B_{n-k}(x_1,\dots,x_{n-k})x_{k+1},\quad n\geq 0.\tag*{\qed}
\]

\begin{rem}
It is interesting to compare the recursion formula \eqref{rec1} with the one obtained by Yu.I.\,Manin (and essentially equivalent to the one of S.\,Keel):
\beq\label{rec2}
\chi(\overline{\eu M}_{0,n+2})=\chi(\overline{\eu M}_{0,n+1})+\sum_{\substack{i+j=n+1\\ i\geq 2}}\binom{n}{i}\chi(\overline{\eu M}_{0,i+1})\chi(\overline{\eu M}_{0,j+1}), \quad n\geq 1.
\eeq
While Manin's formula is quadratic in the $\chi$-terms, our formula \eqref{rec1} is linear. Moreover, in the r.h.s.\,\,of \eqref{rec1} the term $\chi(\overline{\eu M}_{0,n})$ never appears as it is multiplied by $\thi_2=0$: this is in contrast with Manin formula \eqref{rec2}.\qrem
\end{rem}

\subsection{Refinement of the Manin--Zagier asymptotic formula} In \cite{Man95,Man99}, Yu.\,I.~Manin also briefly addressed the growth of the
sequence $\bigl(\chi(\overline{\eu M}_{0,n+1})\bigr)_{n\geq 2}$.
In particular, as a consequence of the main result of \cite{Man95}, he mentions
the following asymptotic estimate, attributed\footnote{Manin writes: \guillemotleft Don Zagier has shown me how to derive from this the following
asymptotical formula [\dots]\guillemotright, see \cite[pag.\,403]{Man95}} to D.\,Zagier:
\[
\chi(\overline{\eu M}_{0,n+1})
\sim
\frac{1}{\sqrt n}
\left(\frac{n}{e^2-2e}\right)^{n-\frac12},
\qquad n\to\infty .
\]
As a final result, we refine this estimate by giving the complete asymptotic
expansion. To state it, we first introduce three real sequences
$(a_k)_{k\geq 0}$, $(\mu_k)_{k\geq 0}$, and $(\omega_k)_{k\geq 0}$, and a polynomial sequence $(C_k(\alpha))_{k\geq 0}$.

The first sequence $(a_k)_{k\geq 0}$ consists of the coefficients of the
Stirling asymptotic series
$n!\sim (n/e)^n\sqrt{2\pi n}\sum_{j\geq 0} a_j n^{-j}$.
They are explicitly given by
\beq
\label{stircoeffs}
a_j
=
\frac{1}{2^j j!}
\left[
\frac{d^{2j}}{dt^{2j}}
\left(
\frac{t^2}{2(e^t-1-t)}
\right)^{j+\frac12}
\right]_{t=0},
\qquad j\geq 0 .
\eeq
For alternative formulas, see \cite[pag.267]{Com74}\cite[Sec.\,2]{BC99}.

The second sequence $(\mu_k)_{k\geq 0}$ is defined by the standard local
expansion of the branch $W_{-1}$ of the Lambert $W$-function near its algebraic
branch point $x=-1/e$:
\beq\label{Wexpansion}
W_{-1}(x)
=
\sum_{\ell\geq 0}\mu_\ell p^\ell
=
-1+p-\frac13 p^2+\frac{11}{72}p^3+\cdots ,
\qquad
p=-\sqrt{2(ex+1)},\quad \operatorname{Im}(x)\geq 0 .
\eeq
Equivalently, the coefficients $\mu_k$ are recursively determined by
$\mu_0=-1$, $\mu_1=1$, and, for $k\geq 2$,
\[
\mu_k
=
\frac{k-1}{k+1}
\left(
\frac{\mu_{k-2}}{2}
+
\frac{\alpha_{k-2}}{4}
\right)
-
\frac{\alpha_k}{2}
-
\frac{\mu_{k-1}}{k+1},
\]
where $\alpha_0=2$, $\alpha_1=-1$, and
$\alpha_k=\sum_{j=2}^{k-1}\mu_j\mu_{k+1-j}$ for $k\geq 2$. For an overview of the properties, applications, and history of $W$, see \cite{CGHJK96}. Various series expansions for $W$, including \eqref{Wexpansion}, are discussed in \cite{CJK97}.

Set $\rho=e-2$ and $A=(2\rho/e)^{1/2}$. The sequence
$(\omega_k)_{k\geq 0}$ is defined by
\beq\label{omk1}
\omega_0=e,\qquad
\omega_1=-eA=-\sqrt{2e\rho},\qquad
\omega_2
=
\sum_{j=1}^{2}
\mu_j(-A)^j\omega_{2-j}
-\rho,
\eeq
and, for $\ell\geq 3$, by
\beq\label{omk2}
\omega_\ell
=
\sum_{j=1}^{\ell}
\mu_j(-A)^j\omega_{\ell-j}.
\eeq

Finally, 
for fixed complex parameter \(\alpha\), we define the polynomials
\(C_k(\alpha)\) by the asymptotic expansion
$$
\frac{\Gamma(n-\alpha)}{\Gamma(n+1)}
\sim
n^{-\alpha-1}
\sum_{k=0}^{\infty}\frac{C_k(\alpha)}{n^k},
\qquad n\to+\infty .
$$
Thus \(C_0(\alpha)=1\), and \(C_k(\alpha)\) is a polynomial in
\(\alpha\) of degree \(2k\) for \(k\geq 1\). The first terms are
$$
C_1(\alpha)=\frac{\alpha(\alpha+1)}{2},\qquad
C_2(\alpha)=
\frac{\alpha(\alpha+1)(3\alpha^2+7\alpha+2)}{24},\qquad
C_3(\alpha)=
\frac{\alpha^2(\alpha+1)^2(\alpha+2)(\alpha+3)}{48}.
$$
See Appendix \ref{appC} for closed and recursive formulas and further properties of the polynomials $C_k(\alpha)$.

The following theorem will be proved in Section \ref{secproofthmMZ+}.
\begin{thm}\label{thmMZ+}
As $n\to\infty$, one has
\[
\chi(\overline{\eu M}_{0,n+1})
\sim
\left(\frac{n}{e^2-2e}\right)^n
\sum_{m=1}^{\infty}\frac{b_m}{n^m},
\]
where, for $m\geq 1$,
\[b_m
=
\sqrt{2}
\sum_{\ell=0}^{m-1}\sum_{k=0}^{\ell}
\frac{(2k+2)!\,a_{m-1-\ell}\,
\omega_{2k+1}}
{(-4)^{k+1}(k+1)!}
C_{\ell-k}\left(k+\frac12\right).
\]
\end{thm}

In particular, the first terms are
\beq\label{newexpasym}
\chi(\overline{\eu M}_{0,n+1})
\sim
\left(\frac{n}{e\rho}\right)^n
\frac{\sqrt{e\rho}}{n}
\left(
c_0+\frac{c_1}{n}
+\frac{c_2}{n^2}
+\frac{c_3}{n^3}+\frac{c_4}{n^4}
+O\left(\frac{1}{n^5}\right)
\right),\qquad c_j=\frac{b_{j+1}}{\sqrt{e\rho}},\quad j\geq 0,
\eeq
where
\[
c_0=1,\qquad
c_1
=
\frac{6e-1}{12e},\qquad
c_2
=
\frac{84e^2-24e-23}{288e^2},\qquad
c_3
=
\frac{9720e^3-3654e^2-1602e-11237}{51840e^3},
\]\[
c_4=\frac{-2482411+672288 e-129168 e^2-143424 e^3+320976 e^4}{2488320 e^4}.
\]

\begin{example}
One has
\[\chi(\overline{\eu{M}}_{0,51})=740385733269564698527983077728073936234668970201726909380777306213753.\]The Manin--Zagier estimate gives
\[\frac{1}{\sqrt{50}}\left(\frac{50}{e^2-2e}\right)^{50-\frac{1}{2}}
\approx 7.33427\times 10^{68}.
\]
Taking the first five correction terms in the refined asymptotic expansion \eqref{newexpasym} gives
\[\frac{1}{\sqrt{50}}\left(\frac{50}{e^2-2e}\right)^{50-\frac{1}{2}}\left(1+\sum_{j=1}^4\frac{c_j}{50^j}\right)\approx 7.4038573310\times 10^{68},
\]with relative error $1.5\cdot 10^{-8}$. \qetr
\end{example}

\section{Proofs of the main results}\label{secproofs}

\subsection{The Getzler--Manin theorem}

Let
\beq\label{Fgen}
F(x,t)
=
x+\sum_{n=2}^{\infty}\frac{x^n}{n!}
{\sf P}_{\overline{\eu M}_{0,n+1}}(t)
\in \Q[t][\![x]\!]
\eeq
be the exponential generating function of the Poincar\'e polynomials of the
spaces \(\overline{\eu M}_{0,n+1}\).

\begin{thm}[Getzler--Manin {\cite{Get95,Man95}}]\label{GMthm}
The series \(F(x,t)\) is the unique solution \(y(x,t)\in \Q[t][\![x]\!]\) of
the Cauchy problem
\beq
\frac{\partial y}{\partial x}
=
\frac{1+y}{1+t^2x-t^2y},
\qquad
y(0,t)=0 .
\label{diffeqgen}
\eeq
Equivalently, \(F(x,t)\) is characterized, as a formal power series in \(x\)
with polynomial coefficients in \(t\), by
\beq
(1+y)^{t^2}
=
t^4y-t^2(t^2-1)x+1,
\qquad
y(0,t)=0 .
\label{funeqgen}
\eeq
\end{thm}

\begin{proof}[Equivalence of the two formulations]
We prove only the equivalence between \eqref{diffeqgen} and
\eqref{funeqgen}. The calculation is formal; one may work first over
\(\Q(t)\) and then observe that the resulting identities have coefficients in
\(\Q[t]\).

Assume first that \(y\) satisfies \eqref{funeqgen}. Differentiating with
respect to \(x\) gives
\[
t^2(1+y)^{t^2-1}\frac{\partial y}{\partial x}
=
t^4\frac{\partial y}{\partial x}
-
t^2(t^2-1).
\]
Using \eqref{funeqgen} to eliminate \((1+y)^{t^2}\), this rearranges to
\eqref{diffeqgen}.

Conversely, suppose that \(y\) satisfies \eqref{diffeqgen}. Since the right
hand side has constant term \(1\), we have \(y=x+O(x^2)\), hence \(y\) admits a
compositional inverse \(x=x(y,t)\). Rewriting \eqref{diffeqgen} in terms of
this inverse gives the linear differential equation
\[
\frac{\partial x}{\partial y}
-
\frac{t^2}{1+y}x
=
\frac{1-t^2y}{1+y}.
\]
Multiplication by the integrating factor
\(
\mu(y,t)=(1+y)^{-t^2}
\)
gives
\[
\frac{\partial}{\partial y}
\left((1+y)^{-t^2}x(y,t)\right)
=
\frac{1-t^2y}{(1+y)^{1+t^2}}.
\]
Using the initial condition \(x=0\) at \(y=0\), we obtain
\[
(1+y)^{-t^2}x
=
\int_0^y
\frac{1-t^2\omega}{(1+\omega)^{1+t^2}}
\,{\rm d}\omega .
\]
A direct integration yields
\[
x
=
\frac{t^4y+1-(1+y)^{t^2}}{t^2(t^2-1)},
\]which is equivalent to \eqref{funeqgen}.
\end{proof}

Manin proved \eqref{funeqgen} by interpreting the ordinary Poincar\'e
polynomials of the smooth projective varieties
\(\overline{\eu M}_{0,n+1}\) as virtual Poincar\'e polynomials. The latter are
compatible with cut-and-paste relations in the Grothendieck ring of varieties.
This makes it possible to express \(F(x,t)\) as a sum over stable trees; the
functional equation then follows from the corresponding combinatorial
``Feynman rules'' for such tree expansions. See also \cite[Ch.\,IV,\S 4]{Man99}.

Getzler obtained the same identities from a different, operadic viewpoint. In
particular, he proved the compositional inverse relations
\beq
\label{GM}
F(G(x,t),t)=G(F(x,t),t)=x,
\eeq
where
\[
G(x,t)
:=
x-\sum_{n=2}^{\infty}\frac{x^n}{n!}
\sum_{i=0}^{n-2}
(-1)^i t^{2(n-i-2)}
\dim_{\C}H_i(\eu M_{0,n+1},\C)
\]
is given explicitly\footnote{This is an immediate consequence of the classical formula
for the Poincar\'e polynomial of the open stratum,
$
\sum_{i=0}^{n-2}
\dim_{\C}H_i(\eu M_{0,n+1},\C)\,u^i
=
\prod_{k=2}^{n-1}(1+ku).
$} by
\beq
\label{Ggenfunc}
G(x,t)
=
\frac{1+t^4x-(1+x)^{t^2}}{t^4-t^2}.
\eeq 

\begin{cor}
If 
\beq\label{gfunc}
g(x,t):=\frac{G(x,t)}{x}=\frac{1+t^4x-(1+x)^{t^2}}{x(t^4-t^2)},
\eeq then
\beq\label{magic2}
{\sf P}_{\overline{\eu M}_{0,n+1}}(t)=
\left.\frac{\der^{n-1}}{x^{\,n-1}}\right|_{x=0}g(x,t)^{-n}.
\eeq
\end{cor}
\proof
Consider the identity \eqref{GM}. Since $G(0,t)=0$ and $\der_x|_{0} G\neq 0$, by the Lagrange inversion formula, we have
\[\![x^{n}]F(x,t)=\frac{1}{n!}\left[\frac{d^{n-1}}{dw^{n-1}}\left(\frac{w}{G(w,t)}\right)^n\right]_{w=0}=\frac{1}{n}[w^{n-1}]\left(\frac{w}{G(w,t)}\right)^n.\tag*{\qed}
\]

\subsection{A simple proof of Aluffi--Marcolli--Nascimento formula}\label{AMNsec}
Setting for convenience $w=t^2$, and multiplying both sides of \eqref{magic2} by $(1-w)^{-n}$,
we have 
\begin{align*}
(1-w)^{-n}&{\sf P}_{\overline{\eu M}_{0,n+1}}(w^{1/2})
=\left.\frac{\der^{n-1}}{x^{\,n-1}}\right|_{x=0}\Bigl((1-w)g(x,w^{1/2})\Bigr)^{-n} \\
&=\left.\frac{\der^{n-1}}{x^{\,n-1}}\right|_{x=0}\left(\frac{-w^2x+(1+x)^w-1}{wx}\right)^{-n}
=\left.\frac{\der^{n-1}}{x^{\,n-1}}\right|_{x=0}\left(1-w+\sum_{k=2}^{\infty}\binom{w}{k}\frac{x^k}{wx}\right)^{-n}\\
&=\left.\frac{\der^{n-1}}{x^{\,n-1}}\right|_{x=0}\sum_{j\ge0}(-1)^j\binom{n+j-1}{j}
\left[-w+\sum_{k=2}^{\infty}\binom{w}{k}\frac{x^k}{wx}\right]^j\\
&=\left.\frac{\der^{n-1}}{x^{\,n-1}}\right|_{x=0}\sum_{j\ge0}\sum_{i=0}^j
(-1)^j\binom{n+j-1}{j}\binom{j}{i}(-w)^{j-i}
\left(\sum_{k=2}^{\infty}\binom{w}{k}\frac{x^k}{wx}\right)^i\\
&=\left.\frac{\der^{n-1}}{x^{\,n-1}}\right|_{x=0}\sum_{j\ge0}\sum_{i=0}^j\sum_{p_1,\dots,p_i=2}^{\infty}
(-1)^j\binom{n+j-1}{j}\binom{j}{i}(-w)^{j-i}
\left(\prod_{r=1}^i\binom{w}{p_r}\right)\frac{x^{p_1+\cdots+p_i}}{(wx)^i}\\
&=(n-1)!\sum_{j\geq 0}\sum_{i=0}^j\sum_{\substack{p_1,\dots,p_i\geq 2\\\sum_{r=1}^ip_r=n+i-1}}(-1)^{i}\binom{n+j-1}{j}\binom{j}{i}\left(\prod_{r=1}^i\binom{w}{p_r}\right)w^{j-2i}.
\end{align*}
\begin{lem}
For integers \(n\ge 1\), \(i\ge 1\), and an indeterminate \(w\), one has
\[
\sum_{\substack{p_1,\dots,p_i\ge 2\\ p_1+\cdots+p_i=n+i-1}}
\prod_{r=1}^i \binom{w}{p_r}
=
\sum_{\substack{a,b,c\ge 0\\ a+b+c=i}}
\binom{i}{a,b,c}
(-1)^b (-w)^c
\binom{wa}{n+i-1-c}.
\]
\end{lem}

\begin{proof}
Consider the generating function
\(
\sum_{p\ge 2} \binom{w}{p} x^p
=
(1+x)^w-1-wx.
\)
Then
\[
\left((1+x)^w-1-wx\right)^i
=
\sum_{m\ge 0}
\left(
\sum_{\substack{p_1,\dots,p_i\ge 2\\ p_1+\cdots+p_i=m}}
\prod_{r=1}^i \binom{w}{p_r}
\right)x^m.
\]
Hence,
\[
\sum_{\substack{p_1,\dots,p_i\ge 2\\ p_1+\cdots+p_i=n+i-1}}
\prod_{r=1}^i \binom{w}{p_r}
=
[x^{n+i-1}]
\left((1+x)^w-1-wx\right)^i.
\]
Expanding multinomially,
\[
\left((1+x)^w-1-wx\right)^i
=
\sum_{\substack{a,b,c\ge 0\\ a+b+c=i}}
\binom{i}{a,b,c}
(1+x)^{wa}(-1)^b(-wx)^c.
\]
Extracting the coefficient of \(x^{n+i-1}\), we get
\(
[x^{n+i-1}](1+x)^{wa}x^c
=
\binom{wa}{n+i-1-c},
\)
and the claim follows. 
\end{proof}

We deduce
\begin{multline*}
(1-w)^{-n}{\sf P}_{\overline{\eu M}_{0,n+1}}(w^\frac{1}{2})=\sum_{j\geq 0}\sum_{i=0}^j\sum_{\substack{a,c \geq 0 \\ a+c \leq i}}(-1)^{a}\frac{(n+j-1)!}{(j-i)!\,a!\,c!\,(i-a-c)!}\binom{wa}{n+i-1-c}w^{j-2i+c}\\
=\sum_{j\geq 0}\sum_{i=0}^j\sum_{\substack{a,c \geq 0 \\ a+c \leq i}}
\sum_{r=0}^{n+i-1-c}
(-1)^{a}
\frac{(n+j-1)!\,s(n+i-1-c,r)\,
a^r}{(j-i)!\,a!\,c!\,(i-a-c)!\,(n+i-1-c)!}
\, 
w^{j-2i+c+r}\\
=\sum_{j\geq 0}\sum_{i=0}^j\sum_{c=0}^i
\sum_{r=0}^{n+i-1-c}
(-1)^{i-c}
\frac{(n+j-1)!}{(j-i)!\,c!\,(n+i-1-c)!}
\, s(n+i-1-c,r)\, S(r,i-c)\,
w^{j-2i+c+r},
\end{multline*}
where in the last equality we have invoked the formula
\(S(n,k)=\frac{(-1)^k}{k!}\sum_{j=0}^k(-1)^j\binom{k}{j}j^n
\), see \cite[Ch.\,V, Thm.\,A, pag.\,204]{Com74}.
By setting \( k = i - c \), so that \( c = i-k \) and \(0 \le k \le i\), we have
\begin{multline*}
(1-w)^{-n}{\sf P}_{\overline{\eu M}_{0,n+1}}(w^\frac{1}{2})=\\
=\sum_{j\ge0} \sum_{i=0}^j \sum_{k=0}^i \sum_{r=0}^{n+k-1}
(-1)^k
\frac{(n+j-1)!\, s(n+k-1,r)\, S(r,k)}{(j-i)!\,(i-k)!\,(n+k-1)!}
\, 
w^{j - i - k + r}\\
=\sum_{j\ge0} (n+j-1)!
\sum_{k\ge0} \sum_{r=0}^{n+k-1}
\frac{(-1)^k s(n+k-1,r) S(r,k)}{(n+k-1)!}
\sum_{i=k}^j \frac{1}{(j-i)!(i-k)!}
w^{j - i - k + r}\\
\end{multline*}
\begin{multline*}
=\sum_{j\ge0} (n+j-1)!
\sum_{k\ge0} \sum_{r=0}^{n+k-1}
\frac{(-1)^k s(n+k-1,r) S(r,k)}{(n+k-1)!}
\frac{(1+w)^{j-k}}{(j-k)!}
\, w^{r-k}\\
=\sum_{j\ge k} \frac{(n+j-1)!}{(j-k)!} (1+w)^j
\sum_{k\ge0} \sum_{r=0}^{n+k-1}
\frac{(-1)^k s(n+k-1,r) S(r,k)}{(n+k-1)!}
(1+w)^{-k}
\, w^{r-k}\\
=(1+w)^k \sum_{t\ge0} \frac{(n+k+t-1)!}{t!} (1+w)^t\sum_{k\ge0} \sum_{r=0}^{n+k-1}
\frac{(-1)^k s(n+k-1,r) S(r,k)}{(n+k-1)!}
(1+w)^{-k}
\, w^{r-k}.
\end{multline*}
Using the identity
\(
\sum_{t\ge0} \frac{(A+t)!}{t!} x^t
= A!\,\frac{1}{(1-x)^{A+1}},
\)
with $A = n+k-1$ and $x = 1+w$, we get
\begin{multline*}
(1-w)^{-n}{\sf P}_{\overline{\eu M}_{0,n+1}}(w^\frac{1}{2})=(n+k-1)! (1+w)^k \frac{1}{(1-(1+w))^{n+k}}\\
\times\sum_{k\ge0} \sum_{r=0}^{n+k-1}
\frac{(-1)^k s(n+k-1,r) S(r,k)}{(n+k-1)!}
(1+w)^{-k}
\, w^{r-k}\\
=\sum_{k\ge0} \sum_{r=0}^{n+k-1}
(-1)^k s(n+k-1,r) S(r,k)
\frac{w^{r-k}}{(-w)^{n+k}}\\
=(-1)^n
\sum_{k\ge0} \sum_{r=0}^{n+k-1}
s(n+k-1,r)\, S(r,k)\,
w^{r-n-2k}.
\end{multline*}
This proves Theorem \ref{AMNthm}.

\subsection{Proof of Theorem \ref{mainthm1}}\label{secproof1} 
The following lemma relates the function $g(x,t)$ defined in equation \eqref{gfunc} with the $\pi_k$ polynomials \eqref{pik}.
\begin{lem}
We have
\[g(x,t)=
1-\sum_{k=1}^\infty\frac{\pi_k(t)}{k!}x^k.
\]
\end{lem}
\proof
A simple computation shows that 
\beq\label{magic1}
\frac{1+w^2x-(1+x)^{w}}{x(w^2-w)}=1-\frac{1}{2}\sum_{k=1}^\infty\binom{w}{2}^{-1}\binom{w}{k+1}x^k=1-\sum_{k=1}^\infty\frac{(k-1)!}{(k+1)!}\binom{w-2}{k-1}x^k.
\eeq
This follows from the binomial series expansion $(1+x)^w=\sum_{i=0}^\infty\binom{w}{i}x^i$ (see Appendix \ref{appB}). The substitution $w=t^2$ implies the result.
\endproof

By formulas \eqref{magic2} and \eqref{coefformexp}, we have
\begin{multline*}
{\sf P}_{\overline{\eu M}_{0,n+1}}(t)=(n-1)!\sum_{m=1}^{n-1}\binom{-n}{m}\sum_{\substack{i_1+\cdots+i_m=n\\ i_j\ge 1}}\prod_{j=1}^m\left(-\frac{\pi_{i_j}(t)}{i_j!}\right)\\
=(n-1)!\sum_{m=1}^{n-1}\cancel{(-1)^m}\binom{n+m-1}{m}\sum_{\substack{i_1+\cdots+i_m=n\\ i_j\ge 1}}\cancel{(-1)^m}\prod_{j=1}^m\left(\frac{\pi_{i_j}(t)}{i_j!}\right)
\end{multline*}
\[
=\sum_{m=1}^{n-1}\frac{(n+m-1)!}{(n-1)!}B_{n-1,m}(\pi_1(t),\pi_2(t),\dots)
\]
where the last equality follows from Lemma \ref{lemmBell}. This proves the theorem.

\subsection{Proof of Theorem \ref{mainthm2}}\label{proofbetti1} Using the standard expansion of the (exponential) partial Bell polynomials,
\[
B_{N,M}(x_1,x_2,\dots)=
\sum_{\substack{j_1,j_2,\dots\ge 0\\
\sum_{k\ge 1} j_k=M\\
\sum_{k\ge 1} k\,j_k=N}}
\frac{N!}{\prod_{k\ge 1} j_k!\,(k!)^{j_k}}
\prod_{k\ge 1} x_k^{j_k},
\]
and substituting $N=n-1$ and $x_k=\pi_k(t)$, one obtains the explicit finite sum
\begin{equation}\label{eq:P_jk}
{\sf P}_{\overline{\eu M}_{0,n+1}}(t)
=
\sum_{\substack{j_1,\dots,j_{n-1}\ge 0\\ \sum_{k=1}^{n-1} k\,j_k=n-1}}
\frac{\bigl(n+\sum_{k=1}^{n-1} j_k-1\bigr)!}{\prod_{k=1}^{n-1} j_k!\,(k!)^{j_k}}
\;\prod_{k=1}^{n-1}\pi_k(t)^{j_k}.
\end{equation}
Expanding the product in \eqref{eq:P_jk} using multinomial coefficients yields
\[
\Bigl[t^{2r}\Bigr]\prod_{k=1}^{n-1}\Bigl(\sum_{a=0}^{k-1} \bt_{k,a}t^{2a}\Bigr)^{j_k}
=
\sum_{\substack{r_{k,a}\ge 0\\
\sum_{a=0}^{k-1} r_{k,a}=j_k\;\forall k\\
\sum_{k=1}^{n-1}\sum_{a=0}^{k-1} a\,r_{k,a}=r}}
\prod_{k=1}^{n-1}
\frac{j_k!}{\prod_{a=0}^{k-1} r_{k,a}!}
\prod_{a=0}^{k-1} \bt_{k,a}^{\,r_{k,a}}.
\]
Plugging this into \eqref{eq:P_jk} (and canceling the factors $j_k!$) gives the fully explicit finite-sum expression
\begin{multline*}
b_{2r}(\overline{\eu M}_{0,n+1})=\Bigl[t^{2r}\Bigr]{\sf P}_{\overline{\eu M}_{0,n+1}}(t)\\
=
\sum_{\substack{j_1,\dots,j_{n-1}\ge 0\\ \sum_{k=1}^{n-1} k\,j_k=n-1}}
\frac{\bigl(n+\sum_{k=1}^{n-1} j_k-1\bigr)!}{\prod_{k=1}^{n-1}(k!)^{j_k}}
\sum_{\substack{r_{k,a}\ge 0\\
\sum_{a=0}^{k-1} r_{k,a}=j_k\;\forall k\\
\sum_{k=1}^{n-1}\sum_{a=0}^{k-1} a\,r_{k,a}=r}}
\prod_{k=1}^{n-1}\prod_{a=0}^{k-1}\frac{\bt_{k,a}^{\,r_{k,a}}}{r_{k,a}!}.
\end{multline*}

\subsection{Proof of Theorem \ref{betti2}}\label{proofbetti2}
Recall the function 
\[g(x,t)=\frac{1+t^4x-(1+x)^{t^2}}{x(t^4-t^2)},
\]introduced in equation \eqref{gfunc}.
\begin{lem}
We have
\[g(z,w^\frac{1}{2})=1-\frac{1}{z}\int_0^z\int_0^u(1+v)^{w-2}dvdu.
\]
\end{lem}
\proof
Write equation \eqref{magic1} as follows
\[\frac{1+w^2z-(1+z)^{w}}{z(w^2-w)}=1-\sum_{k=1}^\infty\frac{1}{k(k+1)}\binom{w-2}{k-1}z^k=1-\frac{1}{z}\sum_{k=1}^\infty\frac{1}{k(k+1)}\binom{w-2}{k-1}z^{k-1}.
\]The series in the last summand equals the double integral above.
\endproof

From equation \eqref{magic2}, for all $\ell\geq 0$ we have
\begin{multline*}
b_{2\ell}(\overline{\eu M}_{0,n+1})=[w^\ell]{\sf P}_{\overline{\eu M}_{0,n+1}}(w^\frac{1}{2})=\frac{1}{\ell!}\left.\frac{\der^{n+\ell-1}}{\der w^\ell\der z^{n-1}}\right|_{(w,z)=(0,0)}\left(1-\frac{1}{z}\int_0^z\int_0^u(1+v)^{w-2}dvdu\right)^{-n}\\
=\frac{1}{\ell!}\left.\frac{\der^{n+\ell-1}}{\der w^\ell\der z^{n-1}}\right|_{(w,z)=(0,0)}\sum_{k\geq 0}\binom{n+k-1}{k}\left(\frac{1}{z}\int_0^z\int_0^u(1+v)^{w-2}dvdu\right)^k.
\end{multline*}

\begin{lem}
We have
\[\frac{1}{\ell!}\left.\frac{\der}{\der w^\ell}\right|_{w=0}\left(\frac{1}{z}\int_0^z\int_0^u(1+v)^{w-2}dvdu\right)=1-\frac{1}{z}\sum_{i=1}^{\ell+1}\frac{\log(1+z)^{i}}{i!},\quad \ell\geq 0.
\]
\end{lem}
\proof
The result follows by exchanging the order of differentiation and integration.
\endproof

If we set
\[\phi_\ell(z)=1-\frac{1}{z}\sum_{i=1}^{\ell+1}\frac{\log(1+z)^{i}}{i!},\quad \ell\geq 0, 
\]by applying the Fa\`a di Bruno and Leibnitz formulas we obtain
\begin{multline*}
b_{2\ell}(\overline{\eu M}_{0,n+1})=\frac{1}{\ell!}\left.\frac{\der^{n+\ell-1}}{\der w^\ell\der z^{n-1}}\right|_{(w,z)=(0,0)}\sum_{k\geq 0}\binom{n+k-1}{k}\left(\frac{1}{z}\int_0^z\int_0^u(1+v)^{w-2}dvdu\right)^k\\
=\left.\frac{\der^{n-1}}{\der z^{n-1}}\right|_{z=0}\sum_{m=0}^\ell\sum_{k\geq m}\binom{n+k-1}{k}\frac{k^{\underline{m}}}{\ell!}\phi_0(z)^{k-m}B_{\ell,m}\left(1!\phi_1(z),2!\phi_2(z)\dots\right)\\
=\sum_{m=0}^\ell\sum_{k\geq m}\sum_{i+j=n-1}\binom{n+k-1}{k}\frac{k^{\underline{m}}}{\ell!}\binom{n-1}{i}\left.\frac{\der^{i}\phi_0(z)^{k-m}}{\der z^{i}}\right|_{z=0}\left.\frac{\der^{j}}{\der z^{j}}B_{\ell,m}\left(1!\phi_1(z),2!\phi_2(z)\dots\right)\right|_{z=0}.
\end{multline*}

For any $a,b,c\in\Z_{\geq 0}$, set
\begin{multline*}
\Om(a,b,c):=\\=\sum_{\substack{i_1+\dots+i_{c+2}=a\\ i_1,\dots,i_{c+2}\ge 0}}(-1)^{-\sum_{r=1}^{c+1}i_r}\binom{a}{i_1,\dots,i_{c+2}}
\frac{1}{\prod_{j=1}^{c+1}j!^{i_j}}\frac{\left(\sum_{j=1}^{c+1}ji_j\right)!\, b!}{(\sum_{r=1}^{c+1}i_r+b)!}\, s\left(\sum_{r=1}^{c+1}i_r+b,\, \sum_{j=1}^{c+1}ji_j\right).
\end{multline*}In particular, we have
\[\Om(a,b,0)=\sum_{p=0}^{a}
(-1)^p \binom{a}{p}
\binom{p+b}{p}^{-1} s(p+b,p),\qquad a,b\in\Z_{\geq 0}.
\]

\begin{lem}\label{lemphi}
For all $a,b\in\Z_{\geq 0}$, we have
\[\left.\frac{\der^{b}\phi_c(z)^{a}}{\der z^{b}}\right|_{z=0}=\Om(a,b,c).
\]
\end{lem}

\proof
By the multinomial expansion, we have
\[\phi_c(z)^a=\sum_{i_1+\dots+i_{c+2}=a}\binom{a}{i_1,\dots,i_{c+2}}(-z)^{-\sum_{r=1}^{c+1}i_r}\frac{\log(1+z)^{\sum_{j=1}^{c+1}ji_j}}{\prod_{j=1}^{c+1}j!^{i_j}}.
\]
From the formula
\[\frac{\log(1+z)^k}{k!}=\sum_{n=k}^\infty s(n,k)\frac{z^n}{n!},
\]
we deduce
\[\frac{\der^b}{\der z^b}\left(\frac{\log(1+z)^k}{z^\Dl}\right)=\sum_{n=k}\frac{k!}{n!}s(n,k)(n-\Dl)^{\underline{b}}z^{n-\Dl-b},
\]so that
\[\left.\frac{\der^b}{\der z^b}\left(\frac{\log(1+z)^k}{z^\Dl}\right)\right|_{z=0}=\frac{k!\,b!}{(\Dl+b)!}s(\Dl+b,k).
\]The result follows.
\endproof

\begin{rem}
It is easy to verify that $\Om(a,b,0)=0$ whenever $a>b$. More generally, by Lemma \ref{lemphi}, one can deduce that for all $a,c\in\N$,
\[
\Om(a+A,a+B,c)=0 \quad \text{for any } A>0 \text{ and } B=1,\dots,c.\tag*{\qrem}
\]
\end{rem}

For all $\ell,m,j\in \Z_{\geq 0}$, set
\[\Om^{[1]}(\ell,m,j):=\frac{\ell!}{m!}\sum_{\substack{\sum_{l=1}^{\ell-m+1}h_l=m\\\sum_{l=1}^{\ell-m+1}lh_l=\ell}}\,\sum_{\sum_{l=1}^{\ell-m+1}\si_l=j}\binom{m}{h_1,h_2,\dots}\binom{j}{\si_1,\si_2,\dots}\prod_{p=1}^{\ell-m+1}\Om(h_p,\si_p,p).
\]
\begin{lem}
We have
\[\left.\frac{\der^{j}}{\der z^{j}}B_{\ell,m}\left(1!\phi_1(z),2!\phi_2(z)\dots\right)\right|_{z=0}=\Om^{[1]}(\ell,m,j).
\]
\end{lem}
\proof The result follows from the known identity
\[B_{\ell,m}(1!x_1,2!x_2,3!x_3,\dots)=\frac{\ell!}{m!}\underbrace{\sum_{\sum_{j}ju_j=\ell}\binom{m}{u_1,u_2,\dots,u_{\ell-m+1}}x_1^{u_1}x_2^{u_2}\dots x_{\ell-m+1}^{u_{\ell-m+1}}}_{\text{ordinary partial Bell polynomial } \hat B_{\ell,m}(x_1,x_2,\dots)},
\]
and the general Lebnitz rule.
\endproof
So, we have
\[
b_{2\ell}(\overline{\eu M}_{0,n+1})\\
=\sum_{m=0}^\ell\sum_{k\geq m}\sum_{i+j=n-1}\binom{n+k-1}{k}\frac{k^{\underline{m}}}{m!}\binom{n-1}{i}\Om(k-m,i,0)\Om^{[1]}(\ell,m,j).
\]Notice that the sum over $k$ can be truncated, since $\Om(k-m,i,0)=0$ if $k-m>i$. In conclusion

\[b_{2\ell}(\overline{\eu M}_{0,n+1})=\sum_{m=0}^\ell\sum_{i=0}^{n-1}\sum_{k=m}^{m+i}\binom{k}{m}\binom{n+k-1}{k}\binom{n-1}{i}\Om(k-m,i,0)\Om^{[1]}(\ell,m,n-1-i).
\]

\subsection{Proof of Proposition \ref{integrprop}}\label{proofintegrprop} 
We have an equivalent expression for $\thi_n$, namely
\beq\label{bexplicit}
\thi_n = (n-1)! \sum_{m=1}^{n-1} \binom{n+m-1}{m} \sum_{\substack{i_1+\cdots+i_m=n-1\\ i_j\geq 1}} \prod_{j=1}^m \frac{i_j-1}{(i_j+1)!}.
\eeq
This follows from Lemma \ref{lemmBell}. Fix $m$ and a composition $(i_1,\dots,i_m)$ of $n-1$, and set $k_j := i_j+1 \geq 2$, so that $k_1+\cdots+k_m = n+m-1$. Using $\binom{n+m-1}{m} = \frac{(n+m-1)!}{m!\,(n-1)!}$, the corresponding summand becomes
\[
(n-1)! \binom{n+m-1}{m} \prod_{j=1}^m \frac{k_j-2}{k_j!} = \frac{1}{m!} \binom{n+m-1}{k_1,\dots,k_m} \prod_{j=1}^m (k_j-2).
\]
Letting $m_r$ count how many $k_j$ equal $r$, we have
\[
\frac{1}{m!}\binom{n+m-1}{k_1,\dots,k_m} = \frac{(n+m-1)!}{\prod_{j=1}^m k_j! \prod_r m_r!},
\]
so each summand equals $\frac{(n+m-1)!}{\prod_j k_j! \prod_r m_r!} \prod_j (k_j-2)$. This is a non-negative integer, since it counts partitions of a set of size $n+m-1$ into unlabeled blocks of sizes $k_1,\dots,k_m$, multiplied by $\prod_j(k_j-2)$. The claim follows.

\subsection{Proof of Theorem \ref{mainthm}} \label{proofmainthm}
Introduce the generating function
\[
\chi(z)=1+z+\sum_{n=3}^\infty \chi(\overline{\eu M}_{0,n})\frac{z^{n-1}}{(n-1)!}.
\]We prove that 
\[\chi(z)=\exp\sum_{n=1}^\infty\thi_n\frac{z^n}{n!},
\]and Theorem \ref{mainthm} follows from identities \eqref{bellid},\eqref{bellid2},\eqref{bellid3}.
The function $\chi(z)$ solves the Cauchy problem
\beq\label{deqman}
\frac{dy}{dz}=\frac{y}{2+z-y},\qquad y(0)=1.
\eeq 
Indeed, this follows from the identity
\(
\chi(z)=1+F(z,-1),
\)
where \(F\) is the generating function defined in \eqref{Fgen}, together with the Getzler--Manin equation \eqref{diffeqgen} in Theorem~\ref{GMthm}.

Let \(W_{-1}\) denote the real branch of the Lambert \(W\)-function satisfying
\[
W_{-1}(x)\leq -1
\qquad
\text{for }
-1/e\leq x<0.
\]
The following closed form was stated in~\cite[Sec.\,8]{AMN25}; we include a proof for completeness.

\begin{prop}We have
\[
\chi(z)=\frac{-z-2}{W_{-1}\!\left(-e^{-2}(z+2)\right)}.
\]
\end{prop}

\begin{proof}
Inverting \eqref{deqman}, we have $\frac{dz}{dy}=\frac{2+z-y}{y}$, i.e.\ $\frac{dz}{dy}-\frac{z}{y}=\frac{2-y}{y}$. Multiplying by the integrating factor $\mu(y)=y^{-1}$ gives $\frac{d}{dy}\!\left(\frac{z}{y}\right)=\frac{2-y}{y^2}$, and integrating yields $\frac{z}{y}=-\frac{2}{y}-\log y+C$. The condition $z=0$ at $y=1$ gives $C=2$, hence $z+2=2y-y\log y$. Setting $y=e^u$ we find $e^u(2-u)=z+2$, or $(u-2)e^{u-2}=-e^{-2}(z+2)$, so $u-2=W\!\left(-e^{-2}(z+2)\right)$. Therefore
\[
y=e^u=e^2 e^{W\left(-e^{-2}(z+2)\right)}=\frac{-z-2}{W_{-1}\!\left(-e^{-2}(z+2)\right)},
\]
where the branch $W_{-1}$ is chosen so that $y(0)=1$.
\end{proof}

We can therefore write \(\chi(z)\) in the implicit form
\[
\chi(z)=e^{u(z)},\qquad
z=\Phi(u(z)),\qquad
\Phi(u)=(2-u)e^u-2.
\]Since $\Phi(0)=0$ and $\Phi'(0)\neq 0$, we can solve the equation $z=\Phi(u)$ in series $u=\sum_{n\geq 1}{a_n}z^n$. The coefficients $a_n$ are given by
\[a_n=\frac{1}{n!}\left[\frac{d^{n-1}}{dt^{n-1}}\left(\frac{t}{\Phi(t)}\right)^n\right]_{t=0},\quad n\geq 1,
\]by the Lagrange inversion formula.

\medskip

We have 
\[\frac{\Phi(t)}{t}=\sum_{k=0}^\infty \al_kt^k,\quad \text{where }\al_k:=-\frac{k-1}{(k+1)!},\quad k\geq 0.
\]
Hence $a_1=1$, while for $n\geq 2$
\begin{multline*}a_n=\frac{1}{n}[t^{n-1}]\left(\frac{t}{\Phi(t)}\right)^n=\frac{1}{n}\sum_{m=1}^{n-1}\binom{-n}{m}\sum_{\substack{i_1+\dots+i_m=n-1\\ i_j\geq 1}}\al_{i_1}\dots\al_{i_m}\\=\frac{1}{n}\sum_{m=1}^{n-1}(-1)^m\binom{n+m-1}{m}\sum_{\substack{i_1+\dots+i_m=n-1\\ i_j\geq 1}}(-1)^m\prod_{j=1}^m\frac{i_j-1}{(i_j+1)!}\\=\frac{1}{n}\sum_{m=1}^{n-1}\binom{n+m-1}{m}\sum_{\substack{i_1+\dots+i_m=n-1\\ i_j\geq 1}}\prod_{j=1}^m\frac{i_j-1}{(i_j+1)!}.
\end{multline*}
Here we have used the standard formula for the coefficient of a generalized binomial series (see Appendix \ref{appB}).
We have
\[\thi_n=n! a_n,\quad n\geq 1,
\]as it follows from equation \eqref{bexplicit}. This completes the proof of Theorem \ref{mainthm}.

\subsection{Proof of Theorem \ref{thmMZ+}}\label{secproofthmMZ+} Let $\rho=e-2$, and recall the coefficients $\om_k$ defined in \eqref{omk1},\eqref{omk2}.

\begin{lem}\label{lemchiexprho}
The radius of convergence of the series
\(
\chi(z)=1+z+\sum_{n=3}^{\infty}\chi(\overline{\eu M}_{0,n})
\frac{z^{n-1}}{(n-1)!}
\)
is $\rho$. Moreover, $z=\rho$ is the unique singularity of $\chi$ on the boundary of its disk of convergence. In a neighborhood of $z=\rho$, the function admits the Puiseux expansion
\begin{equation}\label{chiexprho}
\chi(z)
=
\sum_{\ell\ge0}
\omega_\ell
\left(1-\frac{z}{\rho}\right)^{\ell/2}.
\end{equation}
\end{lem}

\proof
Recall that the branch $W_{-1}(x)$ of the Lambert $W$-function has a logarithmic branch point at $x=0$ and an algebraic branch point at $x=-1/e$. Since
\[
\chi(z)=
-\frac{z+2}
{W_{-1}\!\left(-\frac{z+2}{e^2}\right)},
\]
the singularities of $\chi$ arise from those of $W_{-1}$ under the affine change of variable
\(
x=-\frac{z+2}{e^2}.
\)
The branch point $x=-1/e$ corresponds to
\(
z=e-2=\rho,
\)
while the logarithmic branch point $x=0$ corresponds to
\(
z=-2.
\)
Since $|\rho|=e-2<2$, the closest singularity to the origin is $z=\rho$. Hence the radius of convergence of $\chi$ is $\rho$. Furthermore, the circle $|z|=\rho$ contains no other singularities, so $z=\rho$ is the unique singularity on the boundary of the disk of convergence.

The Puiseux expansion \eqref{chiexprho} follows from the local expansion of $W_{-1}$ at its algebraic branch point $x=-1/e$ (see \eqref{Wexpansion}) after the above change of variable. Indeed, substituting \eqref{Wexpansion} and \eqref{chiexprho} into the identity
\[
W_{-1}\!\left(-\frac{z+2}{e^2}\right)\chi(z)=-z-2
\]
and comparing coefficients yields the recurrence relations
\eqref{omk1} and \eqref{omk2}.
\endproof

Since the function $\chi(z)$ is analytic in a dented neighbourhood of
$z=\rho$, and $\rho$ is its unique dominant singularity, the
Flajolet--Odlyzko--Sedgewick Transfer Theorem
\cite{FO90}\cite[Ch.\,VI.4]{FS09} applies to the singular expansion of
Lemma~\ref{lemchiexprho}.

For $\alpha\notin\{0,1,2,\ldots\}$, the exact coefficient formula
\[
[z^n]\left(1-\frac{z}{\rho}\right)^\alpha
=
\rho^{-n}
\frac{\Gamma(n-\alpha)}
{\Gamma(-\alpha)\Gamma(n+1)}
\]
together with the asymptotic expansion of the quotient of Gamma
functions gives
\[
\frac{\Gamma(n-\alpha)}{\Gamma(n+1)}
\sim
n^{-\alpha-1}
\sum_{r\geq 0}\frac{C_r(\alpha)}{n^r},
\qquad C_0(\alpha)=1,
\]
where the $C_r(\alpha)$ are polynomials in $\alpha$. In particular,
\[
\frac{\Gamma(n-\alpha)}{\Gamma(n+1)}
\sim
n^{-\alpha-1}
\left(
1+
\frac{\alpha(\alpha+1)}{2n}
+
\frac{\alpha(\alpha+1)(\alpha+2)(3\alpha+1)}
{24n^2}
+\cdots
\right).
\]

Only the non-integral powers in the Puiseux expansion contribute to the
asymptotic expansion of the coefficients. Writing the singular part of
the expansion of $\chi$ as
\[
\chi(z)
\sim
\sum_{k\geq 0}
\omega_{2k+1}
\left(1-\frac{z}{\rho}\right)^{k+\frac12},
\]
we therefore obtain
\[
[z^n]\chi(z)
\sim
\rho^{-n}
\sum_{k\geq 0}
\frac{\omega_{2k+1}}{\Gamma\left(-k-\frac12\right)}
n^{-k-\frac32}
\sum_{r\geq 0}
\frac{C_r\left(k+\frac12\right)}{n^r}.
\]
Equivalently, after collecting terms of the same order,
\[
[z^n]\chi(z)
\sim
\rho^{-n}
\sum_{m\geq 0}
\frac{d_m}{n^{m+\frac32}},
\]
where
\[
d_m
=
\sum_{k=0}^{m}
\frac{\omega_{2k+1}}
{\Gamma\left(-k-\frac12\right)}
C_{m-k}\left(k+\frac12\right).
\]

Consequently,
\[
\chi(\overline{\eu M}_{0,n+1})
=
n![z^n]\chi(z)
\sim
n!\rho^{-n}
\sum_{m\geq 0}
\frac{d_m}{n^{m+\frac32}}.
\]

Using Stirling's expansion
\[
n!
\sim
\left(\frac{n}{e}\right)^n
\sqrt{2\pi n}
\sum_{j\geq 0}\frac{a_j}{n^j},
\qquad
\text{$a_j$ as in \eqref{stircoeffs},}
\]
we obtain
\[
\chi(\overline{\eu M}_{0,n+1})
\sim
\left(\frac{n}{e\rho}\right)^n
\sqrt{2\pi}
\sum_{j,m\geq 0}
\frac{a_jd_m}{n^{j+m+1}}.
\]
Collecting the terms with fixed $N=j+m+1$ yields
\[
\chi(\overline{\eu M}_{0,n+1})
\sim
\left(\frac{n}{e\rho}\right)^n
\sum_{N\geq 1}\frac{b_N}{n^N},\qquad b_N
=
\sqrt{2\pi}
\sum_{m=0}^{N-1}
a_{N-1-m}d_m.
\]
Substituting the expression for $d_m$, and invoking the identity
\[
\Gamma\left(-k-\frac12\right)
=
\frac{(-4)^{k+1}(k+1)!}{(2k+2)!}\sqrt{\pi},
\]
we may write
\[
b_N
=
\sqrt{2}
\sum_{m=0}^{N-1}
a_{N-1-m}
\sum_{k=0}^{m}
\omega_{2k+1}
\frac{(2k+2)!}
{(-4)^{k+1}(k+1)!}
C_{m-k}\left(k+\frac12\right).
\]
This gives the complete asymptotic expansion.

\appendix
\section{Bell polynomials}\label{appbell}
General references are \cite{Bel28,Bel34,Com74}.

\medskip

The (exponential) \emph{partial Bell polynomials} $B_{n,k}(x_1,\dots,x_{n-k+1})$ are defined by
\begin{equation}
B_{n,k}(x_1,\dots,x_{n-k+1})
=
\sum_{\substack{k_1 + \cdots + k_{n-k+1} = k \\[2pt]
k_1 + 2k_2 + \cdots + (n-k+1)k_{n-k+1} = n}}
\frac{n!}{k_1! \cdots k_{n-k+1}!}
\prod_{j=1}^{n-k+1} \left(\frac{x_j}{j!}\right)^{k_j}.
\end{equation}
Equivalently, $B_{n,k}$ admits the combinatorial representation
\begin{equation}\label{eqbellpart}
B_{n,k}(x_1,\dots,x_{n-k+1})
=
\sum_{\pi \in \Pi_{n,k}} \prod_{B \in \pi} x_{|B|},
\end{equation}
where $\Pi_{n,k}$ denotes the set of partitions of an $n$-element set into exactly $k$ blocks, and $|B|$ is the cardinality of a block $B$.

Their exponential generating function is given by
\begin{equation}
\frac{1}{k!}
\left( \sum_{m \ge 1} x_m \frac{t^m}{m!} \right)^k
=
\sum_{n \ge k} B_{n,k}(x_1,\dots,x_{n-k+1}) \frac{t^n}{n!}.
\end{equation}

The following result easily follows from the identities above.

\begin{lem}\label{lemmBell}
Let $n\ge m\ge 1$. The following identities hold:
\begin{multline*}
[t^{\,n}]\left(\sum_{k\ge 1} a_k t^k\right)^m
=\frac{m!}{n!}\,
B_{n,m}\big(1!a_1,\,2!a_2,\,3!a_3,\dots\big)
=\sum_{\substack{i_1+\cdots+i_m=n\\ i_j\ge 1}}
\prod_{j=1}^m a_{i_j}\\
=\sum_{\substack{k_1,k_2,\dots\ge 0\\ \sum k_j=m\\ \sum j k_j=n}}
\frac{m!}{k_1!k_2!\cdots}\prod_{j\ge 1} a_j^{k_j}
=\sum_{\substack{\lambda \vdash n\\ \ell(\lambda)=m}}
\frac{m!}{z_\lambda}\prod_{j\ge 1} a_j^{m_j},
\end{multline*}
where $\lambda=(1^{m_1}2^{m_2}\cdots)$, $\ell(\lambda)=\sum_{j\ge 1} m_j$, and
$z_\lambda=\prod_{j\ge 1} j^{m_j}m_j!$.
\qed
\end{lem}
\medskip

The (exponential) \emph{complete Bell polynomials} $B_n(x_1,\dots,x_n)$ are obtained by summing over $k$:
\begin{equation}
B_n(x_1,\dots,x_n)
=
\sum_{k=1}^{n} B_{n,k}(x_1,\dots,x_{n-k+1}).
\end{equation}
Equivalently, they admit the explicit form
\begin{equation}
B_n(x_1,\dots,x_n)
=
\sum_{\substack{k_1 + 2k_2 + \cdots + n k_n = n}}
\frac{n!}{k_1! k_2! \cdots k_n!}
\prod_{j=1}^{n} \left(\frac{x_j}{j!}\right)^{k_j},
\end{equation}
and the representation in terms of set partitions
\begin{equation}
B_n(x_1,\dots,x_n)
=
\sum_{\pi \in \mathcal{P}(n)} \prod_{B \in \pi} x_{|B|},
\end{equation}
where $\mathcal{P}(n)$ denotes the set of all partitions of an $n$-element set and $|B|$ is the cardinality of a block $B$.

Their exponential generating function is
\begin{equation}\label{bellid}
\exp\left( \sum_{k \ge 1} x_k \frac{t^k}{k!} \right)
=
\sum_{n \ge 0} B_n(x_1,\dots,x_n)\frac{t^n}{n!}.
\end{equation}

Moreover, $B_n(x_1,\dots,x_n)$ admits the determinant representations
\begin{align}\label{bellid2}
B_{n}(x_{1},\dots,x_{n})
&=\det\begin{pmatrix}
x_{1} & \binom{n-1}{1}x_{2} & \binom{n-1}{2}x_{3} & \binom{n-1}{3}x_{4} & \cdots & \cdots & x_{n} \\
-1   & x_{1}               & \binom{n-2}{1}x_{2} & \binom{n-2}{2}x_{3} & \cdots & \cdots & x_{n-1} \\
0    & -1                  & x_{1}               & \binom{n-3}{1}x_{2} & \cdots & \cdots & x_{n-2} \\
0    & 0                   & -1                  & x_{1}               & \cdots & \cdots & x_{n-3} \\
0    & 0                   & 0                   & -1                  & \cdots & \cdots & x_{n-4} \\
\vdots & \vdots            & \vdots              & \vdots              & \ddots & \ddots & \vdots \\
0    & 0                   & 0                   & 0                   & \cdots & -1     & x_{1}
\end{pmatrix}\\\label{bellid3}
\end{align}
\begin{align}
&=\det\begin{pmatrix}
\frac{x_{1}}{0!} & \frac{x_{2}}{1!} & \frac{x_{3}}{2!} & \frac{x_{4}}{3!} & \cdots & \cdots & \frac{x_{n}}{(n-1)!} \\
-1   & \frac{x_{1}}{0!} & \frac{x_{2}}{1!} & \frac{x_{3}}{2!} & \cdots & \cdots & \frac{x_{n-1}}{(n-2)!} \\
0    & -2                  & \frac{x_{1}}{0!}               & \frac{x_{2}}{1!} & \cdots & \cdots & \frac{x_{n-2}}{(n-3)!} \\
0    & 0                   & -3                  & \frac{x_{1}}{0!}               & \cdots & \cdots & \frac{x_{n-3}}{(n-4)!} \\
0    & 0                   & 0                   & -4                  & \cdots & \cdots & \frac{x_{n-4}}{(n-5)!} \\
\vdots & \vdots            & \vdots              & \vdots              & \ddots & \ddots & \vdots \\
0    & 0                   & 0                   & 0                   & \cdots & -(n-1)     & \frac{x_{1}}{0!}
\end{pmatrix}
\end{align}

\section{Generalized binomial series}\label{appB}
Let $A$ be a commutative $\Q$-algebra and let
\[
f(x)=1+\sum_{j\ge 1}c_j x^j\in A[\![x]\!].
\]
For any $k\in A$, define
\[
f(x)^k:=\sum_{m\ge 0}\binom{k}{m}\bigl(f(x)-1\bigr)^m\in A[\![x]\!],
\qquad
\binom{k}{m}:=\frac{k(k-1)\cdots(k-m+1)}{m!}.
\]
Writing $f(x)^k=\sum_{n\ge 0}c_n^{(k)}x^n$, one has $c_0^{(k)}=1$ and, for $n\ge 1$,
\beq\label{coefformexp}
c^{(k)}_n=\sum_{m=1}^{n}\binom{k}{m}\!\!\sum_{\substack{i_1+\cdots+i_m=n\\ i_j\ge 1}} c_{i_1}\cdots c_{i_m}.
\eeq
Moreover, the same construction applies to any $k(t)\in A[\![t]\!]$: setting
\[
f(x)^{k(t)}:=\exp \bigl(k(t)\,\log f(x)\bigr)\in A[\![t,x]\!],
\]
one has
\[
f(x)^{k(t)}=\sum_{m\ge 0}\binom{k(t)}{m}\bigl(f(x)-1\bigr)^m \quad\text{in }A[\![t,x]\!],
\]
hence the coefficient of $x^n$ is still given by the same formula with $\binom{k}{m}$ replaced by $\binom{k(t)}{m}$.
In particular this includes the case $k(t)=t$.

\begin{rem}
With the definition $f(x)^{k(t)}:=\exp \bigl(k(t)\log f(x)\bigr)\in A[\![t,x]\!]$ (for $f\in 1+xA[\![x]\!]$, $k(t)\in A[\![t]\!]$, $A$ a commutative $\Q$-algebra), the usual exponent rules hold in $A[\![t,x]\!]$: for $k_1,k_2\in A[\![t]\!]$,
\[
f^{k_1+k_2}=f^{k_1}f^{k_2},\qquad (f^{k_1})^{k_2}=f^{k_1k_2},
\]
and for $g\in 1+xA[\![x]\!]$,
\[
(fg)^k=f^k g^k.
\]
Indeed, $\log(fg)=\log f+\log g$ and $\exp(a+b)=\exp(a)\exp(b)$, as here all series commute.
\end{rem}

\section{The $C_k$-polynomials}\label{appC}

For a fixed complex parameter \(\alpha\), we define the coefficient polynomials
\(C_k(\alpha)\) by the asymptotic expansion, as \(n\to+\infty\),
\[
\frac{\Gamma(n-\alpha)}{\Gamma(n+1)}
\sim
n^{-\alpha-1}\sum_{k\geq 0}\frac{C_k(\alpha)}{n^k}.
\]
Here \(C_0(\alpha)=1\), while \(C_k(\alpha)\in\mathbb{Q}[\alpha]\) has degree
\(2k\) for \(k\geq 1\). The first terms are
\[
C_1(\alpha)=\frac{\alpha(\alpha+1)}{2},\qquad
C_2(\alpha)=\frac{\alpha(\alpha+1)(3\alpha^2+7\alpha+2)}{24},\qquad
C_3(\alpha)=\frac{\alpha^2(\alpha+1)^2(\alpha+2)(\alpha+3)}{48}.
\]

A compact way to compute the sequence is obtained by writing
\[
\log\left(
n^{\alpha+1}\frac{\Gamma(n-\alpha)}{\Gamma(n+1)}
\right)
\sim
\sum_{m\geq 1}\frac{L_m(\alpha)}{n^m},
\qquad
L_m(\alpha)=
\frac{(-1)^{m+1}}{m(m+1)}
\left({\sf B}_{m+1}(-\alpha)-{\sf B}_{m+1}(1)\right),
\]
where the ordinary Bernoulli polynomials \({\sf B}_m(x)\) are defined by
\[
\frac{t e^{xt}}{e^t-1}
=
\sum_{m\geq 0}{\sf B}_m(x)\frac{t^m}{m!}.
\]
Equivalently, the polynomials \(C_k(\alpha)\) are determined by
\(C_0(\alpha)=1\) and
\[
C_k(\alpha)=
\frac{1}{k}\sum_{j=1}^k j\,L_j(\alpha)\,C_{k-j}(\alpha),
\qquad k\geq 1.
\]

They also admit the closed form
\[
C_k(\alpha)
=
(-1)^k\frac{(\alpha+1)_k}{k!}\,
{\sf B}^{(-\alpha)}_k(-\alpha),
\]
where \((\alpha+1)_k\) denotes the rising factorial and the generalized
Bernoulli polynomials are defined by
\[
\left(\frac{t}{e^t-1}\right)^\lambda e^{xt}
=
\sum_{k\geq 0}{\sf B}^{(\lambda)}_k(x)\frac{t^k}{k!}.
\]

Finally, for non-negative integers \(m\), the specialization
\(\alpha=-m\) reflects the elementary identity
\(\Gamma(n+m)/\Gamma(n+1)=(n+1)_{m-1}\) for \(m\geq 1\). Hence the
asymptotic expansion truncates: for \(m\geq 1\),
\[
C_k(-m)=0\qquad\text{for all }k\geq m,
\]
while \(C_k(0)=0\) for all \(k\geq 1\).

\end{document}